\documentclass[a4paper, 10pt]{amsart}

\usepackage{amsmath,amssymb}

\newtheorem{theorem}{Theorem}[section]
\newtheorem{lemma}[theorem]{Lemma}
\newtheorem{corollary}[theorem]{Corollary}
\newtheorem{proposition}[theorem]{Proposition}

\theoremstyle{definition}
\newtheorem{example}[theorem]{Example}
\newtheorem{definition}[theorem]{Definition}

\newtheorem{remarks}[theorem]{\bf Remarks}

\newcommand{\N}{\mathbb N}
\newcommand{\Z}{\mathbb Z}
\newcommand{\R}{\mathbb R}
\newcommand{\Q}{\mathbb Q}

\newcommand{\red}{\text{\rm red}}
\newcommand{\equal}{\text{\rm equal}}
\newcommand{\adj}{\text{\rm adj}}
\newcommand{\mon}{\text{\rm mon}}
\newcommand{\DP}{\negthinspace : \negthinspace}

\DeclareMathOperator{\lcm}{lcm} 

 \DeclareMathOperator{\supp}{supp}

\renewcommand{\t}{\, | \,}

\numberwithin{equation}{section}

\begin{document}

\address{Departamento de {\'A}lgebra, Universidad de Granada,
Granada 18071, Espana}

\email{vblanco@ugr.es, pedro@ugr.es}

\address{Institut f\"ur Mathematik und Wissenschaftliches Rechnen \\
Karl-Franzens-Universit\"at Graz \\
Heinrichstra\ss e 36\\
8010 Graz, Austria} \email{alfred.geroldinger@uni-graz.at}

\author{V{\'i}ctor Blanco and Pedro A. Garc{\'i}a-S{\'a}nchez and Alfred Geroldinger}

\thanks{This work was supported by the
{\it Ministerio de Educaci\'on y Ciencia} (Project No. MTM2007-62346),
by the {\it Junta de Andalucia} (Research Group FQM-343), and by
the {\it Austrian Science Fund FWF} (Project No. P21576-N18)}



\begin{abstract}
Arithmetical invariants---such as sets of lengths, catenary and tame degrees---describe the non-uniqueness of factorizations in atomic monoids.
We study these arithmetical invariants by the monoid of relations and by presentations of the involved monoids. The abstract results will be applied to numerical monoids
and to Krull monoids.
\end{abstract}

\title[Semigroup-theoretical characterizations of arithmetical invariants]{Semigroup-theoretical characterizations of \\ arithmetical invariants  with applications to \\ numerical monoids and Krull monoids}

\maketitle

\bigskip
\section{Introduction} \label{1}
\bigskip

Factorization theory describes the non-uniqueness of factorizations into irreducible elements of atomic monoids by arithmetical invariants, and it studies the
relationship between these arithmetical invariants and algebraic invariants of the objects under consideration. Here, an atomic monoid means a commutative cancellative semigroup with unit element such that every non-unit may be written as a product of atoms (irreducible elements), and main examples  are
the multiplicative monoids consisting of the non-zero elements from a noetherian domain. In abstract semigroup theory, minimal relations and presentations are key tools
to describe the algebraic structure of semigroups. Thus, there should be natural connections between the arithmetical invariants of factorization theory and the presentations of the semigroup. However, only first steps have been made so far to unveil these connections and to apply them successfully for  further investigations. We mention two results in this direction (more can be found in the references). In \cite{C-G-L-P-R06}, it was proved that the catenary degree of a monoid allows a description in terms of $\mathcal R$-equivalence classes (see Proposition \ref{4.6} for details). In \cite{Ph13a}, semigroup-theoretical descriptions are used in the study of the arithmetic of non-principal orders in algebraic number fields.

The aim of the present paper is to explore further the connections between arithmetical invariants and semigroup-theoretical invariants, such as the monoid of  relations
and presentations. We discuss central invariants from factorization theory, such as the $\omega$-invariants (Section \ref{3}), the catenary and monotone catenary degrees (Section \ref{4}), the tame degrees (Section \ref{5}), and finally  Section \ref{6} deals with unions of sets of lengths. We provide---in the abstract setting of atomic monoids---new characterizations or new upper bounds (as in  Propositions \ref{3.3}, \ref{5.2}, Corollary \ref{6.4}), and reveal the influence of special presentations to the arithmetic (as in Theorem \ref{5.5}). Throughout, we apply the abstract results to concrete classes of monoids, mainly to numerical monoids and to Krull monoids
(Corollaries \ref{5.6} and \ref{5.7}, Theorem \ref{6.6}). Moreover,
some of the results have relevance from the computational point of view, as they allow to provide explicit algorithms which partly have been implemented in {\tt GAP} (see Example \ref{3.6} or Remarks \ref{5.8}, and \cite{numericalsgps}). A more detailed discussion of the results will be given at the beginning of each section as soon as we have the required
terminology at our disposal.

\bigskip
\section{Preliminaries} \label{2}
\bigskip

We denote by $\N$ the set of positive integers, and we put $\N_0 =
\N \cup \{0\}$. For every $n \in \mathbb N$, we denote by $C_n$ a
cyclic group with $n$ elements. For real numbers $a, b \in \R$, we
set $[a, b] = \{ x \in \Z \mid a \le x \le b \}$. Let $L, L' \subset
\Z$.  We denote by $L+L' = \{a+b \mid a \in L,\, b \in L' \}$ their
{\it sumset}.
Two distinct elements $k, l \in L$ are called {\it adjacent} if $L \cap [\min\{k,l\},\max\{k,l\}]=\{k,l\}$.
A positive integer $d \in \N$ is called a \ {\it distance} \ of $L$ \ if there exist adjacent elements $k,l \in L$ with $d=|k-l|$.  We denote by \ $\Delta (L)$ \ the {\it set of distances} of $L$.
If $\emptyset \ne L \subset \N$, we call
\[
\rho (L) = \sup \Bigl\{ \frac{m}{n} \; \Bigm| \; m, n \in L \Bigr\}
= \frac{\sup L}{\min L} \, \in \Q_{\ge 1} \cup \{ \infty \}
\]
the \ {\it elasticity} \ of $L$, and we set $\rho (\{0\}) = 1$. By a
{\it monoid}, we mean a commutative, cancellative semigroup with
unit element.

\medskip
\centerline{\it Throughout this paper, let $S$ be a monoid.}
\bigskip

We denote by $\mathcal A (S)$ the set of atoms (irreducible
elements) of $S$, by $S^{\times}$ the group of invertible elements,
by $S_{\red} = \{ a S^{\times} \mid a \in S \}$ the associated
reduced monoid of $S$, and by $\mathsf q (S)$ a quotient group of
$S$ with $S \subset \mathsf q (S)$. A submonoid $T \subset S$ is
said to be {\it saturated} if $T = S \cap \mathsf q (T)$ (equivalently, if $a, b \in T$ and $a$ divides $b$ in $S$, then $a$ divides $b$ in $T$). We say
that $S$ is {\it reduced} if $|S^{\times}| = 1$. If not stated
otherwise, we will use multiplicative notation.  Submonoids of
$(\mathbb Z^s, +)$, in particular numerical monoids, will of course
be written additively.

For a set $P$, we denote by $\mathcal F (P)$ the \ {\it free
$($abelian$)$ monoid} \ with basis $P$. Then every $a \in \mathcal F
(P)$ has a unique representation in the form
\[
a = \prod_{p \in P} p^{\mathsf v_p(a) } \quad \text{with} \quad
\mathsf v_p(a) \in \N_0 \ \text{ and } \ \mathsf v_p(a) = 0 \ \text{
for almost all } \ p \in P \,.
\]
We call $|a|= \sum_{p \in P}\mathsf v_p(a)$ the \emph{length} of
$a$ and $\supp (a) = \{p \in P \mid \mathsf v_p (a) > 0\} \subset P$ the {\it support} of $a$.

The free (abelian) monoid \ $\mathsf Z (S) = \mathcal F \bigl(
\mathcal A(S_\red)\bigr)$ \ is called the \ {\it factorization
monoid} \ of $S$,  the unique homomorphism
\[
\pi \colon \mathsf Z (S) \to S_{\red} \quad \text{satisfying} \quad
\pi (u) = u \quad \text{for each} \quad u \in \mathcal A(S_\red)
\]
is called the \ {\it factorization homomorphism} \ of $S$ and
\[
\sim_S \ = \{ (x,y) \in \mathsf Z (S) \times \mathsf Z (S) \mid \pi
(x) = \pi (y) \}
\]
the \ {\it monoid of relations} of $S$. Clearly, we have $\mathsf Z
(S) \cong (\mathbb N_0^{(\mathcal A (S_{\red})}, +)$, and   if $S$
is written additively, then $\mathsf Z (S)$ and $\sim_S$ will be
written additively too.

Let $\sigma \subset \ \sim_S$ be a subset. Then $\sigma^{-1} = \{(x,y)
\in \sigma \mid (y,x) \in \sigma \}$, and $\sigma$ is called a {\it
presentation} of $S$ if the congruence generated by $\sigma$ equals
$\sim_S$ (equivalently, if $(x,y) \in \mathsf Z (S) \times \mathsf Z
(S)$, then $(x,y)\in \ \sim_S$ if and only if there exist $z_0,\ldots,
z_k\in \mathsf Z (S)$ such that $x=z_0$,  $z_k=y$, and, for all $i
\in [1,k]$, $(z_{i-1},z_{i})=(x_{i-1}w_i,x_iw_i)$ with $w_i\in
\mathsf Z (S)$ and $(x_{i-1},x_i)\in \sigma\cup \sigma^{-1}$). A
presentation $\sigma$ is said to be
\begin{itemize}
\item {\it minimal} if no proper subset of $\sigma$ generates $\sim_S$ (see
      \cite[Chapter 9]{Ro-GS99} for characterizations of minimal
      presentations in our setting).

\smallskip
\item {\it generic} if $\sigma$ is minimal and for all $(x, y) \in \sigma$ we have $\supp (xy) = \mathcal A (S_{\red})$.
\end{itemize}
If $S$ has a generic presentation, then $S_{\red}$ is finitely generated and has no primes.

\smallskip
For a subset $S' \subset S$, we set $\mathsf Z (S') = \{ z \in
\mathsf Z (S) \mid \pi (z) \in S'\}$. Let $Z \subset \mathsf Z (S)$ be a subset. We say that an element $x \in
Z$ is {\it minimal} in $Z$ if for all
elements $y \in Z$ with $y \t x$ it follows that $x =
y$. We denote by $\text{\rm Min} \bigl( Z \bigr)$ the
{\it set of minimal elements} in $Z$. Let $x \in
Z$. Since the number of elements $y \in Z$
with $y \t x$ is finite, there exists an $x^* \in \text{\rm Min}
\bigl( Z \bigr)$ with $x^* \t x$.

For $a \in S$, the set
\[
\begin{aligned}
\mathsf Z (a)  & = \mathsf Z ( \{a\}) \subset \mathsf Z (S) \quad
\text{is the \ {\it set of factorizations} \ of \ $a$} \quad
\text{and}
\\
\mathsf L (a) & = \bigl\{ |z| \, \bigm| \, z \in \mathsf Z (a)
\bigr\} \subset \N_0 \quad \text{is the \ {\it set of lengths} \ of
$a$}  \,.
\end{aligned}
\]
By definition, we have \ $\mathsf Z(a) = \{1\}$ and $\mathsf L (a) =
\{0\}$ for all $a \in S^\times$. The monoid $S$ is called {\it
atomic}   if   $\mathsf Z(a) \ne \emptyset$ \ for all \ $a \in S$
(equivalently, every non-unit can be written as a product of atoms),
and it is called {\it factorial} if $|\mathsf Z(a)| = 1$  for all $a
\in S$. If $S$ is reduced and atomic, then the set of atoms
$\mathcal A (S)$ is the uniquely determined minimal generating set
of $S$ (\cite[Proposition 1.1.7]{Ge-HK06a}). We denote by \(
\mathcal{L}(S) = \{\mathsf{L} (a) \mid a \in S \} \) the {\it system
of sets of lengths} of $S$, and by
\[
\Delta (S) \ = \ \bigcup_{L \in \mathcal L (S) } \Delta ( L ) \ \subset \N
\]
the {\it set of distances} of $S$.

For $z,\, z' \in \mathsf Z (S)$, we can write
\[
z = u_1 \cdot \ldots \cdot u_lv_1 \cdot \ldots \cdot v_m \quad
\text{and} \quad z' = u_1 \cdot \ldots \cdot u_lw_1 \cdot \ldots
\cdot w_n\,,
\]
where  $l,\,m,\, n\in \N_0$ and $u_1, \ldots, u_l,\,v_1,
\ldots,v_m,\, w_1, \ldots, w_n \in \mathcal A(S_\red)$ are such that
\[
\{v_1 ,\ldots, v_m \} \cap \{w_1, \ldots, w_n \} = \emptyset\,.
\]
Then $\gcd(z,z')=u_1\cdot\ldots\cdot u_l$, and we call
\[
\mathsf d (z, z') = \max \{m,\, n\} = \max \{ |z \gcd (z, z')^{-1}|,
|z' \gcd (z, z')^{-1}| \} \in \N_0
\]
the {\it distance} between $z$ and $z'$. For subsets $X, Y \subset
\mathsf Z (S)$, we set
\[
\mathsf d (X, Y) = \min \{\mathsf d (x, y) \mid x \in X, y \in Y \}
\in \mathbb N_0 \,,
\]
and thus   $\mathsf d (X, Y) = 0$ if and only if ( $X  \cap Y \ne
\emptyset$ or $X= \emptyset$ or $Y = \emptyset$ ).

\medskip
\noindent {\bf Numerical Monoids.} By a {\it numerical monoid} we
mean a submonoid $S \subset (\mathbb N_0, +)$ such that the
complement $\mathbb N \setminus S$ is finite. The theory of
numerical monoids is presented in  the recent monograph
\cite{Ro-GS09}. The connection to semigroup algebras and to
one-dimensional local domains (in particular, to power series
domains $K[ \negthinspace [ S ] \negthinspace ]$) is documented in
the surveys \cite{Ba-Do-Fo97, Ba06b, Ba09b} (all these are domains
which have received a lot of attention in factorization theory). We
shall make use of this in Section \ref{5}.

Let $S$ be a numerical monoid. Then $S$ is reduced and finitely
generated. Suppose that $\mathcal A (S) = \{n_1, \ldots, n_t\}$ with
$t \in \mathbb N$ and $1 < n_1 < \ldots < n_t$. Then we write $S =
\langle n_1, \ldots, n_t \rangle$, and since $\mathbb N \setminus S$
is finite, it follows that $\gcd (n_1, \ldots, n_t) = 1$. Writing
factorizations of an element $a \in S$ we put the atoms in boldface
in order to distinguish between the atoms and the scalars. Thus $z =
k_1 \boldsymbol{n_1} + \ldots + k_t \boldsymbol{n_t}$, with $k_1,
\ldots , k_t \in \mathbb N_0$, is the factorization of the element
$a = \sum_{i=1}^t k_i n_i \in S$ of length $|z| = k_1 + \ldots +
k_t$. We denote by $\text{\rm Ap}(S,a) = \{s \in S \mid s - a \notin
S\}$ the {\it Ap\'ery set} of $a$ in $S$ (see \cite{Ro-GS09}).

\medskip
\noindent
{\bf Krull monoids.}   The monoid $S$ is called a {\it
Krull monoid} if it satisfies one of the following equivalent
properties (\cite[Theorem 2.4.8]{Ge-HK06a}){\rm \,:}
\begin{itemize}
\item[(a)] $S$ is $v$-noetherian and completely integrally closed,

\item[(b)] $S$ has a divisor theory,

\item[(c)] $S_{\red}$ is a saturated submonoid of a free monoid.
\end{itemize}
The theory of Krull monoids is presented in the monographs
\cite{HK98, Gr01, Ge-HK06a}. Let $S$ be atomic. Clearly, $\sim_S \ \subset \mathsf Z (S) \times \mathsf Z (S)$ is saturated
and hence $\sim_S$ is a Krull monoid by Property (c)   (more on that can be found in \cite[Lemma
11]{Ph11a}), and hence, in particular, it is atomic. Moreover, if $S_{\red}$ is finitely generated, then $\sim_S$ is finitely generated as a saturated submonoid of a
finitely generated monoid (see \cite[Proposition 2.7.5]{Ge-HK06a}).
An integral
domain $R$ is a Krull domain if and only if its multiplicative
monoid $R \setminus \{0\}$ is a Krull monoid, and thus Property (a)
shows that a noetherian domain is Krull if and only if it is
integrally closed.

Main portions of the arithmetic of a Krull monoid---in particular,
all questions dealing with sets of lengths---can be studied in the
associated  monoid of zero-sum sequences over  its class group. To
provide this concept, let $G$ be an additive abelian group,  $G_0
\subset G$ \ a subset and $\mathcal F (G_0)$ the free monoid with
basis $G_0$. According to the tradition of combinatorial number
theory, the elements of $\mathcal F(G_0)$ are called \ {\it
sequences} over \ $G_0$. For a sequence
\[
U = g_1 \cdot \ldots \cdot g_l = \prod_{g \in G_0} g^{\mathsf v_g
(U)} \in \mathcal F (G_0) \,,
\]
we call $\mathsf v_g(U)$ the {\it multiplicity} of $g$ in $U$,
\[
|U|  = l = \sum_{g \in G} \mathsf v_g (U) \in \mathbb N_0 \
\text{the \ {\it length} \ of \ $U$} \,,  \quad \text{and} \quad
\sigma (U)  = \sum_{i = 1}^l g_i \ \text{the \ {\it sum} \ of \ $U$}
\,.
\]
The monoid
\[
\mathcal B(G_0) = \{ U \in \mathcal F(G_0) \mid \sigma (U) =0\}
\]
is called the \ {\it monoid of zero-sum sequences} \ over \ $G_0$,
and Property (c) shows that  $\mathcal B  (G_0)$ is a Krull monoid.
We define the \ {\it Davenport constant} \ of $G_0$ by
\[
\mathsf  D (G_0) = \sup \bigl\{ |U| \, \bigm| \; U \in \mathcal A
\big( \mathcal B (G_0) \big) \bigr\} \in \N_0 \cup \{\infty\} \,,
\]
which is a classical constant in Combinatorial Number Theory (see
\cite{Ga-Ge06b, Ge09a}).

We will use that for  a reduced finitely generated monoid $S$ the
following statements are equivalent (\cite[Theorem
2.7.14]{Ge-HK06a}){\rm \,:}
\begin{itemize}
\item $S$ is a Krull monoid,

\item $S$ is isomorphic to a monoid $\mathcal B (G_0)$ with $G_0
      \subset G$ as above,

\item $S$ is isomorphic to a monoid of non-negative integer
      solutions of a system of linear Diophantine equations.
\end{itemize}

\bigskip
\section{The $\omega$-invariants} \label{3}
\bigskip

\medskip
\begin{definition} \label{3.1}
Let  $S$  be  atomic. For $b \in S $, let \ $\omega (S, b)$ \ denote
the
      smallest \ $N \in \N_0 \cup \{\infty\} $ \ with the following
      property{\rm \,:}
      \begin{enumerate}
      \smallskip
      \item[] For all $n \in \N $ and $a_1, \ldots, a_n \in S $, if
              $b \t a_1 \cdot \ldots \cdot a_n $, then there exists a
              subset $\Omega \subset [1,n] $ such that $|\Omega | \le N $ and
              \[
              b \Bigm| \, \prod_{\nu \in \Omega} a_\nu \,.
              \]
      \end{enumerate}
      Furthermore, we set
      \[
      \omega (S) = \sup \{ \omega (S, u) \mid u \in
      \mathcal A (S) \} \in \N_0 \cup \{\infty\} \,.
      \]
\end{definition}

\smallskip
Let $S$ be atomic. By definition, an element $b \in S$ is a prime element if and only if $\omega (S, b) = 1$, and  $S$ is factorial if and only if $\omega (S) = 1$.
Thus these $\omega$-invariants (together with the associated tame degrees, see in particular Equation \ref{basic3}) measure in particular how far away atoms are from primes (see \cite{Ge-Ha-Le07,
Ge-Ha08a, Ge-Ha08b, Ge-Ka10a}).
An algorithm to compute the $\omega (S,
\cdot)$ values in numerical monoids was recently presented in
\cite{A-C-K-T11}. Here we start by showing that a slight variant of
the property in the definition of  $\omega (S, \cdot)$ does not
change its value.

\medskip
\begin{lemma} \label{3.2}
Let $S$ be  atomic   and $b \in S$. Then $\omega (S, b)$ is the
smallest  \ $N \in \N_0 \cup \{\infty\} $ \ with the following
property{\rm \,:}
      \begin{enumerate}
      \smallskip

      \item[] For all $n \in \N $ and $a_1, \ldots, a_n \in \mathcal A (S) $, if
              $b \t a_1 \cdot \ldots \cdot a_n $, then there exists a
              subset $\Omega \subset [1,n] $ such that $|\Omega | \le N $ and
              \[
              b \Bigm| \, \prod_{\nu \in \Omega} a_\nu \,.
              \]
      \end{enumerate}
\end{lemma}

\begin{proof}
Let $\omega' (S, b)$ denote the smallest integer \ $N \in \N_0 \cup
\{\infty\} $ satisfying the property mentioned in the lemma. We show
that $\omega (S, b) = \omega' (S, b)$. By definition, we have
$\omega' (S, b) \le \omega (S, b)$. (Note, if $b \in S^{\times}$,
then $\omega (S, b) = \omega' (S, b) = 0$).

In order to show that $\omega (S, b) \le \omega' (S, b)$, let $n \in
\N$ and $a_1, \ldots, a_n \in S$ with $b \t a_1 \cdot \ldots \cdot
a_n$. After renumbering if necessary there is an $m \in [0, n]$ such
that $a_1, \ldots, a_m \in S \setminus S^{\times}$ and $a_{m+1},
\ldots, a_n \in S^{\times}$. Then $b \t a_1 \cdot \ldots \cdot a_m$,
and for every $i \in [1, m]$ we pick a  factorization $a_i = u_{i,1}
\cdot \ldots \cdot u_{i, k_i}$ with $k_i \in \N$ and $u_{i_1},
\ldots , u_{i, k_i} \in \mathcal A (S)$. Then there is a subset $I
\in [1,m]$ and, for every $i \in I$, a subset $\emptyset \ne
\Lambda_i \subset [1, k_i]$ such that
\[
|I| \le \sum_{i \in I} |\Lambda_i| \le \omega' (S, b) \quad
\text{and} \quad b \t \prod_{i \in I} \prod_{\nu \in \Lambda_i}
u_{i, \nu}
\]
which implies that $b \t \prod_{i \in I} a_i$.
\end{proof}

\smallskip
The forthcoming characterization of $\omega (S)$ will be easy to prove. But it is useful from a computational point of view, as well as it will be a key ingredient
in the proof of Theorem \ref{5.5}.

\medskip
\begin{proposition} \label{3.3}
Let $S$ be  atomic.
\begin{enumerate}
\item For every  $s \in S$ we have
      \[
      \omega (S, s) = \sup \bigl\{ |x|  \mid  x \in \text{\rm Min} \bigl(
      \mathsf Z (sS) \bigr) \bigr\} \,.
      \]

\smallskip
\item $\omega (S) = \sup \{ |x| \mid        x \in \text{\rm Min} \bigl(
      \mathsf Z (uS) \bigr) \ \text{for some} \ u \in \mathcal A (S) \bigr\}$.
\end{enumerate}
\end{proposition}

\begin{proof}
Obviously, it is sufficient to prove the first statement.
Furthermore, we may assume that $S$ is reduced.

Let $x = \prod_{u \in \mathcal A
(S)} u^{m_u} \in \text{\rm Min} \bigl( \mathsf Z (sS) \bigr)$. Since
$x \in \mathsf Z (sS)$, it follows that $s$ divides $\prod_{u \in
\mathcal A (S)} u^{m_u}$ (in $S$), and since $x$ is minimal in
$\mathsf Z (sS)$, $s$ does not divide a proper subproduct.
Therefore, we get that $\omega (S, s) \ge \sum_{u \in \mathcal A
(S)} m_u = |x|$.

Conversely, let $(m_u)_{u \in \mathcal A (S)} \in \mathbb
N_0^{\mathcal A (S)}$ be such that $s$ divides $\prod_{u \in
\mathcal A (S)} u^{m_u}$ (in $S$). Then $x = \prod_{u \in \mathcal A
(S)} u^{m_u} \in \mathsf Z (sS)$, and there exists some minimal $x^*
= \prod_{u \in \mathcal A (S)} u^{m_u^*} \in \mathsf Z (sS)$ with
$x^* \t x$ (in $\mathsf Z (S)$). Then
\[
s \t \prod_{u \in \mathcal A (S)} u^{m_u^*} \t \prod_{u \in \mathcal
A (S)} u^{m_u} \quad \text{(in $S$)}
\]
and hence $\omega (S, s) \le \sum_{u \in \mathcal A (S)} m_u^* \le
\sup \bigl\{ |y| \mid y \in \text{\rm Min} \bigl( \mathsf Z (sS)
\bigr) \bigr\}$.
\end{proof}

\medskip
\begin{remarks} \label{3.4}~

\smallskip
1. Note that $\omega (S,s)$ is finite for all $s \in S$ not only for
finitely generated monoids, but more generally for all
$v$-noetherian monoids (see \cite[Theorem 4.2]{Ge-Ha08a}).

\smallskip
2. Let $k, r \in \mathbb N_0$ and $n, d_1, \ldots, d_k \in \mathbb
N$. Let $S \subset (\mathbb N_0^n, +)$ be the set of all non-negative
integer solutions $(x_1, \ldots, x_n) \in \mathbb N_0^n$ of the
following system of equations
\[
\begin{array}{ll}
a_{1,1}x_1+\cdots+a_{1,n}x_n & \equiv 0 \mod d_1,\\
 & \vdots \\
a_{k,1}x_1+\cdots+a_{k,n}x_n & \equiv 0 \mod d_k,\\
a_{k+1,1}x_1+\cdots+a_{k+1,n}x_n & = 0,\\
 & \vdots \\
a_{k+r,1}x_1+\cdots+a_{k+r,n}x_n & = 0,\\
\end{array}
\]
where all $a_{i,j}$ are integers. Obviously, $S \subset (\mathbb
N_0^n, +)$ is a submonoid with
\[
\mathsf q (S) \cap \mathbb N_0^n = S \tag{$*$} \,.
\]
Let $G = \mathbb Z/ d_1 \mathbb Z \times \ldots \times \mathbb Z/
d_k \mathbb Z \times \mathbb Z^r$ and
\[
G_0 = \{ (a_{1,i}+d_1\Z, \ldots, a_{k,i}+d_k\Z, a_{k+1, i}, \ldots,
, a_{k+r,i}) \in G \mid i \in [1, n] \} \,.
\]
Then $S$ is obviously isomorphic to $\mathcal B (G_0)$, the monoid
of zero-sum sequences over $G_0$. This (or independently, the fact
that $S \subset (\N_0^n, +)$ is saturated) show that $S$ is a
reduced, finitely generated Krull monoid.  If a finitely generated
Krull monoid is given in that form, then the characterization of
Proposition \ref{3.3} turns out to be extremely useful, as the next
corollary illustrates.
\end{remarks}

\medskip
\begin{corollary} \label{3.5}
Let $S \subset (\mathbb N_0^n, +)$ be a saturated submonoid with
$\mathcal A (S) = \{ \boldsymbol{s_1}, \ldots, \boldsymbol{s_t}\}$,
where $n, t \in \mathbb N$, and let  $A \in M_{n,t} (\Z)$ be the
matrix whose columns are $s_1,\ldots,s_t$. If $s\in S$ and $(x_1,
\ldots, x_t) \in \N_0^t$, then
\[
x_1 \boldsymbol{s_1} + \ldots  + x_t \boldsymbol{s_t} \in \mathsf
Z(s+S) \quad \text{ if and only if} \quad   A \left( \begin{matrix}x_1 \\
\vdots \\  x_t \end{matrix} \right) \geq s \,.
\]
\end{corollary}

\begin{proof}
Observe that, by definition of  $A$, we have for $s$ and $(x_1,
\ldots, x_t)$ as above, that
\[
x_1 \boldsymbol{s_1} + \ldots  + x_t \boldsymbol{s_t} \in \mathsf
Z(s) \quad \text{ if and only if} \quad   A \left( \begin{matrix}x_1 \\
\vdots \\  x_t \end{matrix} \right) = s \,.
\]
If  $x_1 \boldsymbol{s_1} + \ldots  + x_t \boldsymbol{s_t}   \in
\mathsf Z(s+S)$, then there is  some $s'\in S$ such that $x_1
\boldsymbol{s_1} + \ldots  + x_t \boldsymbol{s_t} = s+s' \in S$, and
hence $A (x_1,\ldots,x_t)^{t} = s+s' \ge s$.

Conversely, let $x \in \mathbb N_0^t$ (considered as a column) be
such that $A x \ge s$. Then $s' =  A x \in S$ and $s'-s \in \mathsf
q (S) \cap \mathbb N_0^n = S$. Thus $s'\in s+S$ and  $x_1
\boldsymbol{s_1} + \ldots  + x_t \boldsymbol{s_t} \in \mathsf Z (s')
\subset \mathsf Z(s+S)$.
\end{proof}

\medskip
The following Example \ref{3.6} illustrates how Proposition \ref{3.3} and
Corollary \ref{3.5} can be used to calculate the
$\omega$-invariants.
It was performed by using an algorithm due to E. Contejean and H. Devie
with slack variables (see \cite{Co-De94}; these authors published later a paper to avoid the
use of these extra variables).

\medskip
\begin{example} \label{3.6}
Let $S \subset (\mathbb N_0^3, +)$ be the set of non-negative integer
solutions  of
\[
\begin{matrix}
x + z =0 \mod 2,\\
y + z=0 \mod 2 \,.
\end{matrix}
\]
Then
\[
\mathcal A (S) = \left\{ \left( \begin{matrix}2 \\ 0 \\ 0
\end{matrix} \right),
\left( \begin{matrix}0 \\ 2 \\ 0 \end{matrix} \right), \left(
\begin{matrix}0 \\ 0 \\ 2 \end{matrix} \right), \left(
\begin{matrix}1 \\ 1 \\ 1 \end{matrix} \right) \right\} \,,
\]
and $S$ is isomorphic to $\mathcal B (G_0)$ with $G_0 = (\Z/2\Z
\times \Z/2\Z)  \setminus \{(0, 0) \}$. The set of solutions of
\[
\begin{pmatrix}
2 & 0 & 0 & 1 \\
0 & 2 & 0 & 1 \\
0 & 0 & 2 & 1
\end{pmatrix}
x \geq
\begin{pmatrix}
1 \\
1 \\
1
\end{pmatrix}
\]
is $\{ (0,0,0,1), (1,1,1,0)\}+ \N_0^4$, and thus, by Proposition
\ref{3.3} and Corollary \ref{3.5}, $\omega \bigl(S,(1,1,1)
\bigr)=3$. The set of solutions of
\[
\begin{pmatrix}
2 & 0 & 0 & 1 \\
0 & 2 & 0 & 1 \\
0 & 0 & 2 & 1
\end{pmatrix}
x \geq
\begin{pmatrix}
2 \\
0 \\
0
\end{pmatrix}
\]
is $\{(1,0,0,0),(0,0,0,2)\} + \N_0^4$, whence $\omega
\bigl(S,(2,0,0) \bigr)=2$. By symmetry, we get $\omega
\bigl(S,(0,2,0) \bigr)=2=\omega \big(S,(0,0,2) \big)$. Thus it follows that
$\omega(S)=3$.
\end{example}

\bigskip
\section{The catenary and monotone catenary degrees} \label{4}
\bigskip

\medskip
\begin{definition} \label{4.1}
Let $S$ be  atomic  and $a \in S$.
\begin{enumerate}
\item Let $z,\, z' \in \mathsf Z (a)$ be factorizations of $a$ and $N \in
      \N_{0} \cup \{ \infty\}$. A finite sequence $z_0,\, z_1, \ldots,
      z_k$ in \ $\mathsf Z(a)$ is called an \ $N$-{\it chain of
      factorizations} \ from $z$ to $z'$ \ if $z = z_0$, \ $z' = z_k$ and
      \ $\mathsf d (z_{i - 1}, z_i) \le N$ for every $i \in [1, k]$.
      In addition, the chain is called {\it monotone} if $|z_0| \le
      \ldots \le |z_k|$ or $|z_0| \ge \ldots \ge |z_k|$.

      \smallskip
      \noindent
      If there exists a (monotone) $N$-chain of factorizations from $z$ to $z'$, we
      say that $z$ and $z'$ \ can be \ {\it concatenated} \ by a
      (monotone)
      $N$-chain.

\smallskip
\item We denote by  $\mathsf c (a) \in \N _0 \cup
      \{\infty\}$ (or by $\mathsf c_{\mon} (a)$ resp.)  the smallest $N \in \N _0 \cup \{\infty\}$ such
      that any two factorizations $z,\, z' \in \mathsf Z (a)$ can be
      concatenated by an $N$-chain (or by a monotone $N$-chain).

\smallskip
\item Moreover,
\[
\mathsf c(S) = \sup \{ \mathsf c(b) \mid b \in S\} \in \N_0 \cup
\{\infty\} \quad \text{and} \quad \mathsf c_{\mon} (S) = \sup \{
\mathsf c_{\mon} (b) \mid b \in S\} \in \N_0 \cup \{\infty\} \quad
\,
\]
denote  the \ {\it catenary degree} \ and the \ {\it monotone
catenary degree} of $S$.
\end{enumerate}
\end{definition}

\medskip
Whereas the catenary degree is a classic invariant in factorization
theory, the  monotone catenary degree was introduced only in
\cite{Fo06a}. However, since then  the existence of monotone and of
near monotone chains of factorizations have been investigated in
various aspects (see \cite{Fo-Ge05, Fo-Ha06b, Ge-Gr-Sc-Sc11}). The monotone
catenary degree is always (explicitly or implicitly) studied in a two-step procedure.

\medskip
\begin{definition} \label{4.2}
Let $S$ be  atomic  and $a \in S$.
\begin{enumerate}
\item For $k \in \mathbb Z$, let $\mathsf Z_k (a) = \{ z \in
      \mathsf Z (a) \mid |z| = k\}$ denote the set of factorizations of
      $a$ having length $k$. We define
      \[
      \mathsf c_{\adj}(a)  = \sup \{ \mathsf d \big( \mathsf Z_k (a),
      \mathsf Z_l (a) \big) \mid k, l \in \mathsf L (a) \ \text{are
      adjacent} \} \,
      \]
      and we set
      \[
      \mathsf c_{\adj} (S) = \sup \{ \mathsf c_{\adj} (b) \mid b \in S \} \in \mathbb N_0 \cup
      \{\infty\} \,.
      \]

\smallskip
\item Let $\mathsf c_{\equal} (a)$ denote the smallest $N \in \mathbb N_0
      \cup \{\infty\}$ with the following property:

      \smallskip
      \begin{itemize}
      \item[]For all $z, z' \in \mathsf Z (a)$ with $|z| = |z'|$ there
      exists a monotone $N$-chain concatenating $z$ and $z'$.
      \end{itemize}

      \smallskip
      \noindent
      We call
      \[
      \mathsf c_{\equal} (S) = \sup \{ \mathsf c_{\equal} (b) \mid b \in S \} \in \mathbb N_0 \cup \{\infty\}
      \]
      the {\it equal catenary degree} of $S$.
\end{enumerate}
\end{definition}

\smallskip
\noindent Obviously,  we have
\[
\mathsf c (a) \le \mathsf c_{\mon} (a) = \sup \{ \mathsf c_{\equal}
(a), \mathsf c_{\adj} (a) \} \le \sup \mathsf L (a) \quad \text{for
all} \quad a \in S \,,
\]
and hence
\begin{equation}
\mathsf c (S) \le \mathsf c_{\mon} (S) = \sup \{ \mathsf c_{\equal}
(S), \mathsf c_{\adj} (S) \} \,. \label{basic2}
\end{equation}
It is well-known that the monotone catenary degree $\mathsf c_{\mon}
(S)$ is finite for finitely generated monoids (\cite[Theorem 3.9]{Fo06a}), and hence also for Krull monoids with finite class group, because it is
stable under transfer homomorphisms. Our results will provide a more natural upper bound
for $\mathsf c_{\mon}
(S)$, valid among others finitely generated monoids.
Inequality \ref{basic} will show that there is a canonical chain of inequalities involving the set of distances, the $\omega$-invariants, and the catenary and tame degrees. However, there seems to be  no obvious relationship  between $\mathsf
c_{\mon}(\cdot)$, on the one side and  $\omega (\cdot)$, the tame degree or on their
canonical upper bound (see Proposition \ref{5.2}) on the other side.
We  study these phenomena by  investigating
$\mathsf c_{\equal}(\cdot)$, and  $\mathsf c_{\adj}
(\cdot)$ individually, and summarize our  discussion after Proposition \ref{5.2}.

\medskip
\begin{definition} \label{4.3}
Let $S$ be  atomic. Then
\[
\sim_{S, \equal} = \big\{ (x,y) \in \mathsf Z (S) \times \mathsf Z (S)
\mid \pi (x) = \pi (y) \ \text{and} \ |x| = |y| \big\}
\]
is called the {\it monoid of equal-length relations} of $S$.
\end{definition}

\medskip
\begin{proposition} \label{4.4}
Let $S$ be  atomic.
\begin{enumerate}
\item $\sim_{S, \equal} \ \subset \ \sim_S$ is a saturated submonoid, and hence $\sim_{S, \equal}$ is a Krull monoid.

\smallskip
\item If $S_{\red}$ is finitely generated, then $\sim_{S, \equal}$ is finitely generated.

\smallskip
\item $\mathsf c_{\equal} (S) \le \sup \{ |x| \mid (x,y) \in \mathcal A (\sim_{S, \equal}) \ \text{for some} \ y \in \mathsf Z (S) \}$.

\smallskip
\item For $d \in \Delta (S)$ let $A_d = \big\{ x \in \mathsf Z (S) \big| d+|x| \in \mathsf L ( \pi (x) \big) \big\}$.  Then \newline
      $\mathsf c_{\adj} (S) \le \sup \{ d+|x| \mid x \in \text{\rm Min} (A_d), \,  d \in \Delta (S) \}$.
\end{enumerate}
\end{proposition}

\begin{proof}
1. Obviously, $\sim_{S, \equal}$ is a submonoid of $\sim_S$. In
order to show that it is saturated, let $(x_1, x_2), (z_1, z_2)$ $
\in  \ \sim_{S, \equal}$ be such that $(x_1, x_2)$ divides  $(z_1,
z_2)$ in $\sim_S$. Then there exists $(y_1, y_2) \in \ \sim_S$ such
that $x_1y_1 = z_1$ and $x_2y_2 = z_2$. This implies that $|y_1| =
|z_1| - |x_1| = |z_2| - |x_2| = |y_2|$, and hence $(y_1, y_2) \in \
\sim_{S, \equal}$. Thus $\sim_{S, \equal} \ \subset \ \sim_S$ is a
saturated. Since $\sim_S$ is a Krull monoid by \cite[Lemma
11]{Ph11a}, $\sim_{S, \equal}$ is a Krull monoid by
\cite[Proposition 2.4.4]{Ge-HK06a}.

\smallskip
2. Let $S_{\red}$ be finitely generated. Then $\sim_S$ is finitely
generated (as observed in Section \ref{2}), and hence $\sim_{S, \equal}$ is finitely generated as a
saturated submonoid of a finitely generated monoid
(\cite[Proposition 2.7.5]{Ge-HK06a}).

\smallskip
3. We set $M = \sup \{ |x| \mid (x,y) \in \mathcal A (\sim_{S, \equal}) \ \text{for some} \ y \in \mathsf Z (S)\}$, and have to show that $\mathsf c_{\equal} (a) \le M$
for all $a \in S$. Let $a \in S$ and $z, z' \in \mathsf Z (a)$ with $|z| = |z'|$. Then $(z, z') \in \ \sim_{S, \equal}$, and we consider
a factorization, say
\[
(z, z') = (x_1, x_1') \cdot \ldots \cdot (x_k, x_k') \quad \text{where} \quad (x_i, x_i') \in \mathcal A (\sim_{S, \equal}) \ \text{for all} \ i \in [1, k] \,.
\]
Then
\[
z = z_0, z_1 = x_1'x_2 \cdot \ldots \cdot x_k, \ldots, z_i = x_1' \cdot \ldots \cdot x_i' x_{i+1} \cdot \ldots \cdot x_k, \ldots,
z_k = x_1' \cdot \ldots \cdot x_k' = z'
\]
is an $M$-chain of factorizations from $z$ to $z'$ and with $|z_i| = |z|$ for all $i \in [0, k]$.

\smallskip
4. Let $a \in S$ and $k, l \in \mathsf L (a)$ be adjacent lengths, say $l-k=d \in \Delta (S)$. We pick some $z \in \mathsf Z_k (a)$. Then there exists some
$x \in \text{\rm Min} (A_d)$ such that $x \t z$. If $x' \in \mathsf Z \big( \pi (x) \big)$ with $|x'| - |x| = d$, then $z' = x'(x^{-1}z) \in \mathsf Z_l (a)$ and
$\mathsf d \big(\mathsf Z_k (a), \mathsf Z_l (a) \big) \le \mathsf d (z, z') \le |x|+d$. This shows that $\mathsf c_{\adj} (a) \le \sup \{ d+|x| \mid x \in \text{\rm Min} (A_d), d \in \Delta (S) \}$, and hence the assertion follows.
\end{proof}

\smallskip
\noindent
Proposition \ref{4.4}.2 will allow us to obtain a more explicit finiteness criterion for $\mathsf c_{\adj} (S)$ in Proposition \ref{5.2}.

\medskip
\begin{definition} \label{4.5}
Let $S$ be  atomic.
\begin{enumerate}
\item Two elements $z, z' \in \mathsf Z (S)$ are $\mathcal R$-{\it related} if $z = z' = 1$ or if $z$ and $z'$ can be concatenated by a chain of factorizations
      $z = z_0, \ldots, z_k = z'$ such that $\pi (z_i) = \pi (z)$ and $\gcd (z_{i-1}, z_i) \ne 1$ for all $i \in [1, k]$.

\smallskip
\item For  $a \in S$, we denote by $\mathcal R_a$ the {\it set of $\mathcal R$-$($equivalence$)$ classes} of $\mathsf Z (a)$. For $\sigma \in \mathcal R_a$
      we set $|\sigma| = \min \{|z| \mid z \in \sigma\}$, and we define
      \[
      \mu (a) = \sup \{ |\sigma| \mid \sigma \in \mathcal R_a\} \,.
      \]

\smallskip
\item We set
      \[
      \mu (S) = \sup \{ \mu (a) \mid a \in S \ \text{with} \ |\mathcal R_a| \ge 2 \} \,.
      \]
\end{enumerate}
\end{definition}

We will need the following result, first proved for finitely generated monoids (see \cite[Theorem 3.1]{C-G-L-P-R06})  and then in
the general setting (see \cite[Corollary 9]{Ph11a}).

\medskip
\begin{proposition} \label{4.6}
If $S$ is atomic, then $\mathsf c (S) = \mu (S)$.
\end{proposition}

There is no analogue result for the monotone catenary degree. Below we will provide the first example of a monoid which is not tame but which has finite
monotone catenary degree (indeed the catenary and the monotone catenary degree coincide). On the other side of the spectrum there are tame monoids with infinite monotone catenary degree. We recall the notion of finitely primary monoids, a concept which stems from ring theory. The monoid $S$ is  called  {\it finitely primary}
if there exist $s,\, \alpha \in \N$ with the following properties:
\begin{enumerate}
\item[]
$S$ is a submonoid of a factorial monoid $F= F^\times \times [p_1,\ldots,p_s]$ with $s$
pairwise non-associated prime elements $p_1, \ldots, p_s$ satisfying
\[\qquad
S \setminus S^\times \subset p_1 \cdot \ldots \cdot p_sF \quad \text{and} \quad (p_1 \cdot \ldots \cdot p_s)^\alpha F \subset S  \,.
\]
If this is the case, then we say that $S$ is finitely primary of \ {\it rank} \ $s$ \ and \ {\it exponent} \ $\alpha$.
\end{enumerate}
If $s=1$ and $F^{\times} = \{1\}$, then $S$ is isomorphic to a numerical monoid. Furthermore, $S$ is tame if and only if it is of rank $1$ (\cite[Theorem 3.1.5]{Ge-HK06a}).

\medskip
\begin{example} \label{4.7}
Let $S = (\mathbb N \times \mathbb N \cup \{(0,0)\}, +)$. Then $\mathsf c (S) = \mathsf c_{\mon} (S) = 3$ and hence $\Delta (S) = \{1\}$.
However, we have $\rho (S) = \infty$ and hence $\omega (S) = \mathsf t (S) = \infty$ (for the invariants not defined so far, see the discussion at Inequality \ref{basic}).
\begin{proof}
By definition, $S$ is a finitely primary monoid of rank $2$ and exponent $1$.
It is a special case of the monoid studied in \cite[Example 3.1.8]{Ge-HK06a}, where all assertions have been verified apart from the
formula for $\mathsf c_{\mon} (S)$. Indeed, it is straightforward that $\mathcal A (S) = \{ (1,m), (m,1) \mid m \in \mathbb N\}$. Furthermore, every element $a = (a_1, a_2) \in S \setminus (\mathcal A (S) \cup \{(0,0)\})$ can be written as a sum of two atoms, namely $(a_1, a_2) = (1, a_2-1) + (a_1-1, 1)$. These two observations
easily imply the assertions on $\mathsf c (S), \Delta (S)$ and $\rho (S)$. In order to show that $\mathsf c_{\mon} (S) = 3$ we proceed in two steps.

First we show that $\mathsf c_{\adj} (S) = 3$. Let $a \in S$ and $k, l \in \mathsf L (a)$ be adjacent lengths, say $k < l$. Since $\Delta (S) = \{1\}$, it follows that $l = k+1 \ge 3$. Let $z = u_1 \cdot \ldots \cdot u_{k+1} \in \mathsf Z_l (a)$ with $u_1, \ldots, u_{k+1} \in \mathcal A (S)$. Since $u_1u_2u_3$ is a product of two atoms, say $u_1u_2u_3 = v_1v_2$ with $v_1, v_2 \in \mathcal A (S)$, we infer that $z' = v_1v_2u_4 \cdot \ldots \cdot u_{k+1} \in \mathsf Z_k (a)$, and hence
$3 = \mathsf d (z, z') = \mathsf d \big(\mathsf Z_{k+1}(a), \mathsf Z_k (a) \big)$. Thus have $\mathsf c_{\adj} (a) = 3$, and hence we get $\mathsf c_{\adj} (S) = 3$.

Second, we verify that $\mathsf c_{\equal} (S) = 3$. Let $a = (m, n) \in S$ and $k \in \mathsf L (a)$. We have to show that each two factorizations $z, z' \in \mathsf Z_k (a)$ can be concatenated by a monotone $3$-chain of factorizations. By symmetry, we may assume that $m \le n$, and then we clearly have $k \le m$. We consider the factorization
\[
z^* = (k-2)\boldsymbol{(1,1)} + \boldsymbol{(m-k+1,1)} + \boldsymbol{(1,n-k+1)} \in \mathsf Z_k (a) \,.
\]
Clearly, it is sufficient to show that from every factorization $z \in \mathsf Z_k (a)$ there is a monotone $3$-chain of factorizations to $z^*$. Let $z \in \mathsf Z_k (a)$ be given. We proceed by induction on $\mathsf v_{\boldsymbol{(1,1)}} (z)$. If $\mathsf v_{\boldsymbol{(1,1)}} (z) = k-2$, then $z = z^*$ and we are done.
Suppose that $\mathsf v_{\boldsymbol{(1,1)}} (z) < k-2$. Then there are two atoms $u_1 = \boldsymbol{(1, a_1)}$ and $u_2 = \boldsymbol{(1, a_2)}$ with $a_1, a_2 \in \mathbb N_{\ge 2}$ and $(u_1+u_2) \t z$, or there are two atoms $u_1 = \boldsymbol{(a_1, 1)}$ and $u_2 = \boldsymbol{(a_2, 1)}$ with $a_1, a_2 \in \mathbb N_{\ge 2}$ and $(u_1+u_2) \t z$. By symmetry, we may suppose that the first case holds. Then we define
$z' = -u_1-u_2 + z + \boldsymbol{(1,1)} + \boldsymbol{(1,a_1+a_2-1)}$. Clearly, we have $z' \in \mathsf Z_k (a)$, $\mathsf d (z, z') = 3$, $\mathsf v_{\boldsymbol{(1,1)}} (z) < \mathsf v_{\boldsymbol{(1,1)}} (z')$, and hence the assertion follows.
\end{proof}
\end{example}

\bigskip
\section{The tame degrees} \label{5}
\bigskip

\medskip
\begin{definition} \label{5.1}
Let $S$ be  atomic.
\begin{enumerate}
\item For $a \in S$ and $x \in \mathsf Z (S)$, let $\mathsf t
      (a,x) \in \N_0 \cup \{\infty\}$ denote the
      smallest $N\in \N_0 \cup \{\infty\}$ with the following property{\rm \,:}
      \begin{enumerate}
      \smallskip
      \item[] If  $\mathsf Z(a) \cap x\mathsf Z(S) \ne \emptyset$ and
              $z \in \mathsf Z(a)$, then there exists $z' \in
              \mathsf Z(a) \cap x\mathsf Z(S)$ such that  $\mathsf d (z,z') \le
              N$.
      \end{enumerate}

\smallskip
\item For  subsets $S' \subset S$ and $X \subset \mathsf Z(S)$, we
      define
      \[
      \mathsf t (S',X) = \sup \big\{ \mathsf t (a,x) \, \big| \, a \in S',  x \in
      X\big\} \in \N_0 \cup \{\infty\} \,,
      \]
      and for $a \in S$, we set $\mathsf t (a, X) = \mathsf t ( \{a\},
      X)$.

\smallskip
\item $S$ is said to be {\it locally tame} if $\mathsf t (S, u) <
      \infty$ for all $u \in \mathcal A (S_{\red})$. We call
      \[
      \mathsf t (S) = \mathsf t (S, \mathcal A \big(S_{\red})\big) = \sup \{
      \mathsf t (S, u) \mid u \in \mathcal A (S_{\red}) \} \in \mathbb N_0
      \cup \{\infty\}
      \]
      the {\it tame degree} of $S$, and $S$ is called {\it tame} if $\mathsf t (S) < \infty$.

\smallskip
\item We set ${\mathsf a (S)} = \sup \{ |x| \mid (x,y) \in \mathcal A (
      \sim_S) \ \text{for some} \ y \in \mathsf Z (S) \} \in \N_0 \cup \{\infty\}$.
\end{enumerate}
\end{definition}

\medskip
Let $S$ be  reduced and atomic. Local tameness is a central
finiteness property in factorization theory, but the finiteness of
the tame degree is a rare property (a non-principal order $\mathfrak
o$ in an algebraic number field is locally tame with finite catenary
degree, and it is tame if and only if for every prime ideal
$\mathfrak p$ containing the conductor there is precisely one prime
ideal $\overline{\mathfrak p}$ in the principal order $\overline{
\mathfrak o}$ such that $\overline{\mathfrak p} \cap \mathfrak o =
\mathfrak p$). Whereas in $v$-noetherian monoids (these are monoids
satisfying the ascending chain condition for $v$-ideals) we have
$\omega (S, u) < \infty$ for all atoms $u \in \mathcal A (S)$, this
does not hold for the $\mathsf t (S, u)$ values (see \cite[Corollary
3.6]{Ge-Ha-Le07}, \cite[Theorems 4.2 and 4.4]{Ge-Ha08a},
\cite[Theorems 5.3 and 6.7]{Ge-Ha08b}). However, we have $\omega (S)
< \infty$ if and only if $\mathsf t (S) < \infty$ (see
Inequality \ref{5.1} below).  A main aim in this
section is to show that monoids having a generic presentation
satisfy $\omega (S) = \mathsf t (S)$ (Theorem \ref{5.5} and its
corollaries). After that we provide the first examples of numerical
monoids $S$ with $\omega (S) < \mathsf t (S)$.

Let $S$ be  atomic  and, to avoid trivialities, suppose that $S$ is
not factorial. Let us consider ${\mathsf a (S)}$.
If $S$ is finitely generated, then $\sim_S$ is
finitely generated, hence $\mathcal A ( \sim_S)$ is finite, and thus ${\mathsf a (S)} < \infty$. It
has been proved that ${\mathsf a (S)}$ is an upper bound for a variety of
arithmetical invariants, such as the catenary degree (e.g.,
\cite[Proposition 14]{Ph11a}; see also the forthcoming Corollary
\ref{6.4}). However, we have
\begin{equation}
2 + \sup \Delta (S) \overset{(1)}{\le} \mathsf c (S) \overset{(2)}{\le} \omega (S) \overset{(3)}{\le} \mathsf t
(S) \overset{(4)}{\le} \omega (S)^2 \quad \text{and} \quad  \rho (S) \overset{(5)}{\le} \omega (S) \,, \label{basic}
\end{equation}
where $\rho (S)$ is the
elasticity (see Definition \ref{6.1}): for $(1)$ see \cite[Theorem 1.6.3]{Ge-HK06a}, for $(2), (4)$ and $(5)$ see \cite[Section 3]{Ge-Ka10a}, and
$(3)$ can be found in  \cite[Theorem 3.6]{Ge-Ha08a}.
In the next proposition we will verify that ${\mathsf a (S)}$ is an upper bound for the tame degree
$\mathsf t (S)$, but even this  inequality  can be
strict (even for numerical monoids, see
\cite[Example 4.4]{C-G-L-P-R06}).

\medskip
\begin{proposition} \label{5.2}
Let $S$ be  atomic.
\begin{enumerate}
\item For every $u \in \mathcal A (S_{\red})$ we have
      \[
      \mathsf t (S, u) \le \sup \{ |x|, |y| \mid (x,y) \in \mathcal
      A (\sim_S), \, x \in u \mathsf Z (S) \} \,.
      \]

\smallskip
\item $\mathsf t (S) \le {\mathsf a (S)}$.

\smallskip
\item If \ ${\mathsf a (S)} < \infty$, then $\mathsf c_{\adj} (S) < \infty$.
\end{enumerate}
\end{proposition}

\begin{proof}
1. and 2. Obviously, it is sufficient to prove the first statement.
Furthermore, we may assume that $S$ is reduced.

Let $u \in \mathcal A (S)$, $a \in uS$ and $z \in \mathsf Z (a)$. We
have to find a factorization $z' \in \mathsf Z (a) \cap u \mathsf Z
(S)$ such that
\[
\mathsf d (z, z') \le \sup \{ |x|, |y| \mid (x,y) \in \mathcal
      A (\sim_S), x \in u \mathsf Z (S) \} \,.
\]
Since $a \in uS$, there exists some $\bar{z} \in \mathsf Z (a) \cap
u \mathsf Z (S)$. We consider a factorization of $(z, \bar{z})$ in
$\sim_S$, say
\[
(z, \bar{z}) = (z_1, \bar{z_1}) \cdot \ldots \cdot (z_k, \bar{z_k})
\,,
\]
where $k \in \mathbb N$, $(z_i, \bar{z_i}) \in \mathcal A ( \sim_S)$
for all $i \in [1, k]$ and $\bar{z_1} \in u \mathsf Z (S)$. Then $z'
= \bar{z_1} (z_1^{-1}z) \in \mathsf Z (a) \cap u \mathsf Z (S)$ and
\[
\mathsf d (z, z') \le \max \{|z_1|, |\bar{z_1}|\} \le \sup \{ |x|,
|y| \mid (x,y) \in \mathcal
      A (\sim_S), x \in u \mathsf Z (S) \} \,.
\]

\smallskip
3. Suppose that ${\mathsf a (S)} < \infty$. Then
\[
B = \{ |x| - |y| \mid       (x,y) \in \mathcal A (\sim_S) \} \subset \mathbb Z
\]
is finite. Furthermore, by 2. and Inequality \ref{basic}, the set of distances $\Delta (S)$ is finite. By Proposition \ref{4.4}.4,
it suffices to verify that
\[
\sup \{ |x| \mid x \in \text{\rm Min} (A_d), \,  d \in \Delta (S) \} < \infty \,,
\]
where $A_d = \{ x \in \mathsf Z (S) \mid d+|x| \in \mathsf L \big( \pi (x)\big) \}$. Let $x \in \mathsf Z (S)$ and $y \in \mathsf Z \big( \pi (x)\big)$ such that $|y| = |x| + d$. Consider a factorization
\[
(x,y) = (x_1, y_1) \cdot \ldots \cdot (x_k, y_k)
\,,
\]
where $k \in \mathbb N$ and $(x_i, y_i) \in \mathcal A ( \sim_S)$
for all $i \in [1, k]$. There exists a bound $M(B, d)$ with the following property (for the construction of
an explicit upper bound,  see \cite[Lemma 5.1]{Ge-Gr09b}): there is a subset $I \subset [1,k]$ with $|I| \le M(B, d)$ such that for
\[
x' = \prod_{i \in I} x_i \quad \text{and} \quad y' = \prod_{i \in I} y_i
\]
we have $|y'| = |x'| + d$. Since $x' \t x$, there is an $x^* \in \text{\rm Min} (A_d)$ with $x^* \t x' \t x$ and
$|x^*| \le |x'| \le M(B, d) {\mathsf a (S)}$.
\end{proof}

\medskip
We discuss the relationship between the finiteness of the monotone catenary degree, of the tame degree and of its upper bound ${\mathsf a (S)}$.
Example \ref{4.7} shows that the monotone catenary degree can be finite even if the monoid is not tame. Conversely,
Example 4.5 in \cite{Fo06a} provides a finitely primary monoid of rank $1$ (hence it is tame and has finite catenary degree) for which
$\mathsf c_{\equal}(S)$  is infinite.
In contrast to that example, Theorem 5.1 from \cite{Ge-Ka10a} shows that a  slightly weaker variant of
$\mathsf c_{\adj}(S)$ is finite for tame monoids. More precisely, it states that in a tame monoid $S$ there is a constant $M \in \mathbb N$ with the following property:
\begin{itemize}
\item[] For each two adjacent lengths \ $k, \, l \in \mathsf L (a)
           \cap [\min \mathsf L (a) + M, \, \max \mathsf L (a) - M]$ \ we
           have \newline
           $\mathsf d \bigl( \mathsf Z_k (a), \mathsf Z_l (a) \bigr) \le M$.
\end{itemize}
Proposition \ref{5.2}.3 shows that the finiteness of ${\mathsf a (S)}$---which is stronger than the finiteness of the tame degree---enforces the finiteness of $\mathsf c_{\adj}(S)$. It is an interesting open problem  whether the finiteness of the tame degree is strong enough to guarantee the finiteness of $\mathsf c_{\adj}(S)$.

\medskip
Let $S$ be atomic but not factorial. We discuss the relationship between $\omega (S)$ and $\mathsf t (S)$.
If $u \in S$ is an atom but not
prime, then
\[
\mathsf t (S, uS^{\times}) = \max \bigl\{ \omega (S, u), 1 + \tau
(S, u) \bigr\}  \in \mathbb N_{\ge 2} \cup \{\infty\} \,
\]
and thus
\begin{equation}
\mathsf t (S) = \max \{ \omega (S), 1+ \tau (S) \} \label{basic3}
\end{equation}
(see \cite[Theorem 3.6]{Ge-Ha08a}; since it is not needed in the
sequel, we do not repeat the definition of the $\tau$-invariant).
For a large class of Krull monoids it was found out that $\mathsf t
(S) = 1 + \tau (S)$ (\cite[Corollary 4.6]{Ge-Ha08a}). In contrast to
that result, M. Omidali recently proved that $\omega (S) = \mathsf t
(S)$ for numerical monoids generated by almost arithmetical
progressions (see \cite[Theorem 3.10]{Om10a}).  Theorem \ref{5.5} provides a result of this type in a more general setting. We will frequently make use of
the following fact: if $\mathsf t (S) < \infty$, then there is an $a
\in S$ and a $u \in \mathcal A (S_{\red})$ such that $\mathsf t (a,
u)  = \mathsf t (a, \mathcal A \big(S_{\red}) \big) = \mathsf t (S)$.

\medskip
\begin{lemma} \label{5.3}
Let $S$ be a reduced and finitely generated. Then for every subset
$X \subset S$ there exists  a  finite set $E \subset X$ such that $X
\subset EH$. Clearly, $E$ can be chosen to be minimal.
\end{lemma}

\begin{proof}
This is a special case of \cite[Proposition 2.7.4]{Ge-HK06a}.
\end{proof}

\smallskip
The next lemma is a generalization of Lemma 5 in \cite{C-G-L-P-R06}.

\medskip
\begin{lemma} \label{5.4}
Let $S$ be  reduced and atomic, $a \in S$, $z \in \mathsf Z (a)$
and $u, v \in \mathcal A (S)$.
\begin{enumerate}
\item Suppose that $\mathsf t \big(a, \mathcal A (S) \big) = \mathsf d \big(z,
      \mathsf Z (a) \cap u \mathsf Z (S) \big) > 0$, $z \in v \mathsf Z (S)$
      and $a \in uvS$. Then $\mathsf t \big(v^{-1}a, \mathcal A (S) \big) \ge
      \mathsf t \big(a, \mathcal A (S) \big)$.

\smallskip
\item Let $a \in S$ be minimal such that $\mathsf t \big(a, \mathcal A
      (S) \big) = \mathsf t (S) > 0$ $($this means that no proper divisor $b$ of
      $a$ satisfies $\mathsf t \big(b, \mathcal A (S) \big) = \mathsf t (S)$$)$ and
      let $z' \in \mathsf Z (a) \cap u \mathsf Z (S)$ such that $\mathsf d
      (z, z') = \mathsf t (S)$. Then $z \in \text{\rm Min} \bigl(
      \mathsf Z (uS) \bigr)$.
\end{enumerate}
\end{lemma}

\begin{proof}
1. By definition, we have $\mathsf t \big(v^{-1}a, \mathcal A (S) \big) \ge
\mathsf d \big(v^{-1}z, \mathsf Z (v^{-1}a) \cap u \mathsf Z (S) \big)$. If
$z' \in \mathsf Z (v^{-1}a) \cap u \mathsf Z (S)$ such that
$\mathsf d (v^{-1}z, z') = \mathsf d \big(v^{-1}z, \mathsf Z (v^{-1}a)
\cap u \mathsf Z (S) \big)$, then
\[
\begin{aligned}
\mathsf t \big(v^{-1}a, \mathcal A (S) \big) & \ge \mathsf d (v^{-1}z, z') =
\mathsf d (z, vz') =
\mathsf d \bigl(z, v \bigl( \mathsf Z (v^{-1}a) \cap u \mathsf Z (S) \bigr) \bigr) \\
 & \ge \mathsf d \big(z, \mathsf Z (a) \cap u \mathsf Z (S) \big) = \mathsf t \big(a, \mathcal A (S) \big) \,.
\end{aligned}
\]

\smallskip
2. Assume to the contrary that $z \notin \text{Min} \bigl( \mathsf Z
(uS)  \bigr)$. Then there exists an atom, say $v$, such that
$v^{-1}z \in \mathsf Z (uS)$. Since $z' \in u \mathsf Z (S)$ with
$\mathsf d (z, z') = \mathsf t (S) > 0$, it follows that $u \nmid z$
and hence $\mathsf d \big( z, \mathsf Z (a) \cap u \mathsf Z (S) \big) > 0$.
Thus 1. implies that $\mathsf t \big(v^{-1}a, \mathcal A (S) \big) \ge
\mathsf t \big(a, \mathcal A (S) \big) = \mathsf t (S)$, a contradiction to
the minimality of $a$.
\end{proof}

\medskip
We say that $S$ has a {\it unique minimal presentation} if it has a minimal presentation $\sigma$  and for each minimal presentation $\tau$
we have $\sigma \cup \sigma^{-1} = \tau \cup \tau^{-1}$. If this holds then $\sigma$ is called a {\it unique minimal presentation} of $S$.
We will need that every generic presentation is a unique minimal presentation.
 Although this can be obtained as a consequence of a result by I.Peeva and B.Sturmfels, we present a short and  independent proof in the language of monoids (see \cite[Remark 4.4.3]{Pe-St98}; indeed the term {\it generic presentation} has been chosen to reflect the origins in {\it generic lattice ideals}).

\medskip
\begin{proposition} \label{4.8}
Every generic presentation of $S$ is a  unique minimal presentation.
\end{proposition}

\begin{proof}
We may suppose that $S$ is reduced. Then $S$ is finitely generated. Recall that any
minimal presentation is constructed by choosing pairs of elements in
different $\mathcal R$-classes of elements with more than one $\mathcal
R$-class (see \cite[Chapter 9]{Ro-GS99}).

Let $\sigma \subset \ \sim_S$ be a generic presentation.
Then every pair $(x,y)\in
\sigma$ has full support, and $x$ and $y$ are in different $\mathcal
R$-classes. Thus, if $x, y \in \mathsf Z(s)$, then $\mathsf Z(s)$ can
consist of only two $\mathcal R$-classes, and the union of their support is the set
of all atoms. So for every $s \in S$ with $|\mathsf Z (s)|\ge 2$, the set of factorizations $\mathsf Z (s)$
consists of precisely two $\mathcal R$-classes, and $\sigma$ is unique if and only if  each such $\mathcal R$-class
contains precisely one factorization.

Assume to the contrary that there is an $s\in S$ such that $\mathsf Z(s)$ consists of two
$\mathcal R$-classes and that two distinct factorizations $z,z'\in \mathsf Z(s)$
are in the same $\mathcal R$-class.
By definition, $z$ and $z'$ can be
concatenated by a chain of factorizations
      $z = z_0, \ldots, z_k = z'$ such that $\pi (z_i) = s$ and $\gcd (z_{i-1}, z_i) \ne 1$ for all $i \in [1, k]$.

We set $z_1 = xy_1$, $z_2 = xy_2$ and $s_1 = \pi (y_1) \in S$ where $x = \gcd (z_1, z_2)$ and $y_1, y_2 \in \mathsf Z (S)$. Note that $s_1$ is a proper divisor of $s_0 = s$. Since $\supp (y_1) \subset \supp (z_1)$ and
$\supp (y_2) \subset \supp (z_2)$, $y_1$ and $y_2$ are in the same $\mathcal R$-class, because otherwise $\mathsf Z (s_1)$ would consist of two $\mathcal R$-classes
and there would be a relation without full support.

Iterating this construction we obtain an infinite sequence $(s_i)_{i \ge 0}$ where $s_{i+1}$ is a proper divisor of $s_i$ for all $i \in \mathbb N_0$, a contradiction to $S$ being finitely generated.
\end{proof}

\medskip
\begin{theorem} \label{5.5}
Let $S$ be atomic, $P \subset S$ a set of representatives of the set
of primes of $S$ and $T$ the set of all $a \in S$ such that $p \nmid
a $ for all $p \in P$. Suppose that $T = \coprod_{i \in I} T_i$, $T
\ne T^{\times}$ and that there is an $i^* \in I$ such that $T_{i^*}$
has a generic presentation and $\mathsf t (T_{i^*}) = \mathsf t
(T)$. Then $\mathsf c (S) = \omega (S) = \mathsf t (S)$.
\end{theorem}

\smallskip
\noindent {\it Remark.} Since $\mathsf t (T) = \sup \{\mathsf t
(T_i) \mid i \in I\}$ (\cite[Proposition 1.6.8]{Ge-HK06a}), the
assumption is of course satisfied if all $T_i$ have generic
presentations.

\begin{proof}
By \cite[Theorem 1.2.3]{Ge-HK06a}, $T \subset S$ is an atomic
submonoid and $S = \mathcal F (P) \times T$. Note that neither $S$
nor $T$ are factorial because $T \ne T^{\times}$. Since
\[
\mathsf c (T_{i^*}) \le \mathsf c (T) = \mathsf c (S) \le \omega (S) \le \mathsf t
(S) = \sup \{ \mathsf t (T_i) \mid i \in I \} = \mathsf t (T_{i^*})
\,,
\]
it suffices to show that $\mathsf c (T_{i^*}) = \omega (T_{i^*}) =
\mathsf t (T_{i^*})$. Thus, after a change of notation, we may
assume that $S$ is reduced, not factorial and has a generic
presentation. Let $\sigma \subset \ \sim_S$ denote this generic
presentation of $S$. We start with the following assertion.

\begin{enumerate}
\item[{\bf A.}\,] For every $u \in \mathcal A (S)$ we have
\[
\text{Min} \bigl( \mathsf Z (uS) \bigr) = \{u\} \cup \{ x \in
\mathsf Z (S) \mid (x,y) \in \sigma \cup \sigma^{-1} \ \text{for
some} \ y \in \mathsf Z (uS) \} \,.
\]
\end{enumerate}

\noindent {\it Proof of \,{\bf A}}.\,  Since $u \in   \text{Min}
\bigl( \mathsf Z (uS) \bigr)$, we may focus on the elements
different from $u$. Let $x = \prod_{v \in \mathcal A (S)} v^{k_v}
\in \text{Min} \bigl( \mathsf Z (uS) \bigr) \setminus \{u\}$ with
$k_v \in \mathbb N_0$ for all $v \in \mathcal A (S)$, and set $a =
\pi (x) \in S$. Since $x \ne u$, it follows that $k_u = 0$. There
exists a $y \in \mathsf Z (a) \cap u \mathsf Z (S)$, and the pair
$(x,y)$ belongs to the congruence generated by $\sigma$. Hence there
exists $(x', y') \in \sigma \cup \sigma^{-1}$ with $x' \t x$. Since
$u \nmid x$, we get that $u \nmid x'$ and because $\sigma$ is
generic, it follows that $u \t y'$ and hence $y' \in \mathsf Z
(uS)$. Since $x \in \text{Min} \bigl( \mathsf Z (uS) \bigr)
\setminus \{u\}$, we infer that $x=x'$. The uniqueness property of a generic presentation implies that $\mathsf Z (a) = \{x', y'\}$. Thus we get
that $y=y'$ and hence $(x,y) = (x',y') \in \sigma$.

Conversely, let $x \in \mathsf Z (S)$ and $y \in \mathsf Z (uS)$
such that $(x,y) \in \sigma \cup \sigma^{-1}$. Then we clearly  have
$x \in \mathsf Z (uS)$, and assume to the contrary that it is not
minimal. Then there is an $x' = \prod_{v \in \mathcal A (S)}v^{k_v'}
\in \mathsf Z (uS)$, where $k_v' \in \mathbb N_0$ for all $v \in
\mathcal A (S)$, and  with $x' \t x$ and $x' \ne x$. Note that $k_u'
= 0$. Since $\sigma$ is generic, there exists a $y' = \prod_{v \in
\mathcal A (S)} v^{l_v}$, where all $l_v \in \mathbb N_0$, with $l_u
\ne 0$ and $\pi (x') = \pi (y')$. Since the pair $(x', y')$ is in
the congruence generated by $\sigma$, there exist $(x'', y'') \in
\sigma \cup \sigma^{-1}$ such that $x'' \t x'$. This implies that
$x'' \t x$ and $x'' \ne x$. This contradicts the fact, that
elements whose factorizations appear in a generic
presentation are not comparable (\cite[Corollary 6]{Ga-Oj11a}). \qed

\smallskip
Since $\sigma$ is a generic  presentation, for every $u \in \mathcal
A (S)$ and every $(x,y) \in \sigma$ we have $u \t x$ or $u \t y$. Thus {\bf A} and Proposition \ref{3.3} imply that
\[
\omega (S) = \max \{ \max \{|x|, |y|\} \mid (x,y) \in \sigma \} \,.
\]
Now the minimality property of a generic presentation (see Proposition \ref{4.8}), the above formula
for $\omega (S)$ together with Proposition \ref{4.6} imply that $\omega (S) = \mathsf c (S)$.

Since $\omega (S) \le \mathsf t (S)$, it remains to show that
converse inequality. Let $a \in S$ be minimal such that $\mathsf t \big(a, \mathcal A (S) \big) = \mathsf t (S)$, and let $u \in \mathcal A (S)$, $z \in
\mathsf Z (a)$ and $z' \in \mathsf Z (a) \cap u \mathsf Z (S)$ such
that
\[
\mathsf t (S) = \mathsf t \big(a, \mathcal A (S) \big) = \mathsf d (z, z')
\,.
\]
By Lemma \ref{5.4}, it follows that $z \in \text{Min} \big( \mathsf Z
(uS) \big)$. Thus, by {\bf A}, there exist $(x,y) \in \sigma \cup
\sigma^{-1}$ such that $y=z$. Therefore $x$ and $y$ are factorizations of
$a$, which appear in the unique presentation of $S$. This implies
that $\mathsf Z (a) = \{x,y\}$, and thus $z' = x$ and $\mathsf t (S)
= \mathsf d (z,z') \le \max \{|x|, |y|\} \le \omega (S)$.
\end{proof}

\medskip
Let $R$ be an integral domain. We denote by $R^{\bullet} = R
\setminus \{0\}$ its multiplicative monoid of non-zero elements, by
$\mathfrak X (R)$ the set of all minimal non-zero prime ideals of
$R$, by $\widehat R$ its complete integral closure, and by $(R \DP
\widehat R) = \{ f \in R \mid f \widehat R \subset R \}$ the
conductor of $R$ in $\widehat R$.

\medskip
\begin{corollary} \label{5.6}
Let $R$ be a weakly Krull domain,  $\mathfrak f = (R \DP \widehat
R)\ne \{0\}$,  $\mathcal P^* = \{ \mathfrak p \in \mathfrak X(R)
\mid \mathfrak p \supset \mathfrak f \}$ and $S = \mathcal I_v^*
(R)$ the monoid of $v$-invertible $v$-ideals equipped with
$v$-multiplication. If for every $\mathfrak p \in \mathcal P^*$, the
monoid $R_{\mathfrak p}{}^\bullet$ has a generic presentation, then
$\mathsf c (S) = \omega (S) = \mathsf t (S)$.
\end{corollary}

\begin{proof}
By \cite[Theorem 3.7.1]{Ge-HK06a}, the monoid $S$ is isomorphic to
$\mathcal F (P) \times T$, where
\[
\mathcal P = \{ \mathfrak p \in \mathfrak X(R) \mid \mathfrak p
\not\supset \mathfrak f \} \quad \text{and} \quad T =
\prod_{\mathfrak p \in \mathcal P^*} (R_{\mathfrak
p}{}^\bullet)_{\red}  \,.
\]
Thus the assertion follows from Theorem \ref{5.5}.
\end{proof}

\smallskip
Let all notations be as in Corollary \ref{5.6}. It is easy to point
out explicit examples where the assumptions hold (for details see \cite[Section 3.7]{Ge-HK06a}). Every
one-dimensional noetherian domain, in particular every order in a
Dedekind domain, is weakly Krull. Let $\mathfrak p \in \mathcal
P^*$. Then $R_{\mathfrak p}{}^\bullet$ is finitely primary, and
$R_{\mathfrak p}{}^\bullet$ is tame if and only if there exists
precisely one prime ideal $\widehat{\mathfrak p} \in \mathfrak X(
\widehat R)$ satisfying $\widehat{\mathfrak p} \cap R = \mathfrak
p$. Suppose this holds true, and set $H = (R_{\mathfrak
p}{}^\bullet)_{\red}$. Then $H \subset F = F^{\times} \times [p]$
where $p$ is a prime element of the factorial monoid $F$, and its
value monoid $\mathsf v_p (H) = \{\mathsf v_p (a) \mid a \in H \}
\subset (\mathbb N_0, +)$ is a numerical monoid. If $R$ is a
non-principal order in an algebraic number field, then $F^{\times}$
is finite.

\smallskip
\begin{corollary} \label{5.7}
Let $S$ be a numerical monoid with $\mathcal A (S) = \{ n_1,n_2,n_3
\}$ where $\gcd (n_1, n_2) = \gcd (n_1, n_3) = \gcd (n_2, n_3) = 1$.
Then $S$ has a generic presentation and  $\mathsf c (S) = \omega (S)
= \mathsf t (S)$. More precisely, if
\[
c_i=\min\{ k\in \mathbb N ~|~ kn_i\in \langle n_j,n_k\rangle\} \quad
\text{and} \quad c_in_i=r_{i,j}n_j+r_{i,k}n_k \,,
\]
where $\{i,j,k\}=\{1,2,3\}$ and   $r_{i, j} \in \mathbb N_0$, then $
\mathsf
c(S)=\max\{c_1,c_2,c_3,r_{12}+r_{1,3},r_{2,1}+r_{2,3},r_{3,1}+r_{3,2}\}$.
\end{corollary}

\begin{proof}
$S$ has a generic presentation by \cite[Lemma 10.18]{Ro-GS09} and hence Theorem \ref{5.5} implies that $\mathsf c (S) =
\omega (S) = \mathsf t (S)$. The formula for $\mathsf c (S)$ stems
from \cite[Exercise 8.23]{Ro-GS09} (note that the integers $r_{i,j}$ are uniquely
determined).
\end{proof}

\smallskip
The catenary degree of   numerical monoids with embedding dimension
three has been also described (with a different approach) in
\cite{Ag-Ga09a}. Let $S$ be atomic. Obviously, the requirement that $S$ has a generic presentation (enforcing $\mathsf c (S) = \mathsf t (S)$)
is a strong assumption, and also the general philosophy in factorization theory confirms the idea that the equality of the catenary and the
tame degree should be an exceptional phenomenon (see also Corollary \ref{5.9}). On the other hand,  all types of numerical monoids studied so far
share this exceptional phenomenon. This contrast will become more clear in the following remark, where we also construct the first infinite family of numerical monoids whose catenary degrees are strictly smaller than the tame degrees.

\medskip
\begin{remarks} \label{5.8}~

\smallskip
1. Let  $S$ be a numerical monoid with $\mathcal A (S) =
\{n_1,\ldots,n_t\}$, where $t \in \mathbb N$ and $1 < n_1 < \ldots <
n_t$, and let $s \in S$. Consider a factorization $z=a_1\mathbf n_1+ \cdots +a_t\mathbf
n_t\in {\rm Min} \big(\mathsf Z (s+S) \big)$.
Then $z \in \mathsf Z (a)$ where $a=a_1n_1+\ldots+a_tn_t$, and  $a=s+u$ for
some $u \in S$.
Pick some $j\in [1, t]$ such that
$\mathbf n_j | z$.   The minimality of $z$ implies that $z-\mathbf n_j\not\in
\mathsf Z(s+S)$, and thus $a-s-n_j=u-n_j\not \in S$ whence $u \in
\text{\rm Ap} (S, n_j)$. Thus, if we want to compute the elements in ${\rm
Min}\big(\mathsf Z(s+S) \big)$ (which in view of Proposition \ref{3.3}
enables us to determine $\omega(S)$), we only have to find the
factorizations of the elements of the form $s+u$ with $u$ in the
Ap\'ery set of some atom.

We implemented this procedure in {\tt GAP} by using the {\tt
numericalsgps} package (see \cite{numericalsgps}). We did an
exhaustive search computing all numerical monoids with Frobenius
number up to 20. That makes 3515 numerical monoids, and   the only
monoids $S$ in this set fulfilling $\omega(S)<\textsf t(S)$ are
$\langle 5, 6, 9 \rangle$, $\langle 5, 8, 12 \rangle$ and $\langle
6, 8, 9 \rangle$.

\smallskip
2. The minimal presentations of the above three
numerical monoids are very similar. Playing around with the
Smith normal form of the matrix whose rows are the differences of
the relators of these monoids, one can find even wilder examples.
The monoid $S= \langle 19,46,391\rangle$ has
$\omega(S)=23<39=\mathsf t(S)$.

\smallskip
3. We present an infinite family of numerical monoids whose $\omega$-invariants
are strictly smaller than the tame degrees.  Let $q$ be a prime, $p_1, p_2 \in \mathbb N$ with $p_1<p_2$, $p_1+p_2=q$ and $\gcd(p_1,p_2)=1$,
and $k \in \mathbb N_{\ge 2} \setminus q \mathbb N$ such that $p_1k < q < p_2k$. We define
\[
S_k = \langle p_1k, \, q, \,p_2k \rangle \,,
\]
and set  $n_1=p_1k$, $n_2=q$, $n_3=p_2k$ and $c_i=\min\{ m \in \mathbb N \mid m n_i\in \langle
n_j,n_k\rangle\}$ with $\{i,j,k\}=\{1,2,3\}$ (note that $c_1, c_2$ and $c_3$ are as in Corollary \ref{5.7}).
\begin{enumerate}
\item[(a)] The Diophantine equation $qx+p_2ky=p_2kt$ has general solution $x=kt-p_2ks$, $y=-t+qs$, $s\in \mathbb Z$.
           The first $t$ for which $x$ and $y$ can be non-negative is $t=p_2$.
           This in particular means that $p_1k$ is not in $\langle
           q,p_2k\rangle$ and that $c_1=p_2$. In fact, $p_2n_1=p_1n_3$, and
           $(p_2+1)n_1=kn_2+(p_1-1)n_3$.

\item[(b)] Analogously one proves that $c_2=k$; $kn_2=n_1+n_3$.

\item[(c)] It is also easy to show that $c_3=p_1$: $p_1n_3=p_2n_1$. Moreover, $(p_1+1)n_3=(p_2-1)n_1+kn_2$.
\end{enumerate}
By using this information it easily follows that
\[
\begin{matrix}
{\rm Min}\big(\mathsf Z(n_1+S_k) \big)=\{\mathbf n_1, k\mathbf n_2, p_1\mathbf n_3\},\\
{\rm Min}\big(\mathsf Z(n_2+S_k) \big)=\{(p_2+1)\mathbf n_1, \mathbf n_1+\mathbf n_3, \mathbf n_2, (p_1+1)\mathbf n_3\},\\
{\rm Min} \big(\mathsf Z(n_3+S_k) \big)=\{p_2\mathbf n_1, k\mathbf n_2,
\mathbf n_3\}.
\end{matrix}
\]
Therefore Proposition \ref{3.3} implies that
\[
\omega(S_k)=\max\{k,p_2+1\} \,.
\]
Note that $\mathsf Z \big((p_1+1) n_3 \big)= \{(p_1+1) \mathbf n_3, (p_2-1)\mathbf n_1+k\mathbf n_2, p_2\mathbf n_1+\mathbf n_3\}$.
Thus, analyzing the factorizations of the elements in $\pi \big({\rm
Min}\big(\mathsf Z(n_i+S_k) \big) \big)$ for $i \in [1,3]$, and by using 1., we obtain that
\[
\mathsf t(S_k) = \max \big\{ \mathsf t \big( (p_2+1)e_1, \mathcal A (S_k) \big),
\mathsf t \big((p_1+1)e_3, \mathcal A (S_k) \big) \big\} = \max\{ p_2+1, k+p_2-1\}
= k+p_2-1  \,,
\]
which is strictly larger than $\omega(S_k)$. Furthermore,   if $k \ge p_2+1$, then
\[
1 <  \dfrac{\mathsf t(S_k) }{\omega(S_k)} = 1 + \frac{p_2-1}{k} \le 1 + \frac{p_2-1}{p_2+1} < 2 \,.
\]
\end{remarks}

\medskip
We end this section with a brief glance at Krull monoids. For them
the equivalence of the catenary and the tame degree is an even rarer
phenomenon than it is for numerical monoids. Let $S$ be a Krull
monoid with  class group $G$ and let $G_P \subset G$ denote the set of classes containing prime divisors. If $\mathsf D (G_P) < \infty$,
then $S$ is tame, and the converse holds---among others---if $S$ the multiplicative monoid of non-zero elements of a domain
(see \cite[Theorem 4.2]{Ge-Ka10a}). If $G$ is finite with  $|G| \ge 3$ and $G_P = G$, then \cite[Corollary
3.4.12]{Ge-HK06a} shows that
\[
\mathsf c (S) = \mathsf c (G) \le \mathsf D (G) = \omega (S) \le
\mathsf t (G) \le \mathsf t (S) \,,
\]
where the final inequality can be strict (\cite[Example
3.4.14]{Ge-HK06a}).

\medskip
\begin{corollary} \label{5.9}
Let $G$ be a finite abelian group with $|G| \ge 3$.
\begin{enumerate}
\item $\mathsf c (G) = \mathsf t (G)$ if and only if $G \in \{C_3, C_4, C_2^2, C_2^3\}$.

\smallskip
\item The monoid of zero-sum sequences $\mathcal B (G^{\bullet})$ has a
           generic  presentation if and only if \newline $G  \in \{ C_3, C_2^2 \}$.
\end{enumerate}
\end{corollary}

\begin{proof}
1. See \cite[Corollary 6.5.7]{Ge-HK06a}.

\smallskip
2.  By 1. and by Theorem \ref{5.5}, we have to check only the groups
in $\{C_3, C_4, C_2^2, C_2^3\}$. We recall the following facts (for details see \cite[Chapter 9]{Ro-GS99}). If $\sigma$ is a minimal presentation for $\mathcal B (G^{\bullet})$ and $(a,b)\in \sigma$,
then $a$ and $b$ are in different $\mathcal R$-classes. In fact, any
minimal presentation is constructed by choosing pairs of elements in
different $\mathcal R$-classes of elements with more than one $\mathcal
R$-class.

If $G = C_3 = \{0, g, 2g\}$, then $\mathcal A \big( \mathcal B (G^{\bullet}) \big) = \{  U_1 = g^3, U_2 = (2g)^3, V = g(2g) \}$ and
$\sigma = \{ (U_1U_2, V^3)\} \subset \ \sim_{\mathcal B (G^{\bullet})}$ is a generic presentation.

If $G = C_2 \oplus C_2 = \{0, e_1, e_2, e_1+e_2\}$, then $\mathcal A \big( \mathcal B (G^{\bullet}) \big) = \{  U_1= e_1^2, U_2 = e_2^2, U_3 = (e_1+e_2)^2, V = e_1e_2(e_1+e_2) \}$
and $\sigma = \{(U_1U_2U_3, V^2) \} \subset \ \sim_{\mathcal B (G^{\bullet})}$ is a generic presentation.

Let $G = C_4 = \{0, g, 2g, -g\}$. Then $\mathcal A \big( \mathcal B (G^{\bullet}) \big) = \{ U_1 = g^4, U_2 = (2g)^2, U_3 = (-g)^4, U_4 = (-g)g, U_5 = g^2 (2g), U_6 = ((2g)(-g)^2 \}$
and $(U_1 U_3, U_4^4) \in \ \sim_{\mathcal B (G^{\bullet})}$.
Since $\mathsf Z (U_4^4) = \{ U_1U_3, U_4^4\}$, the set of factorizations of $U_4^4$
has only two
$\mathcal R$-classes, where each consists of precisely one factorization. Thus $(U_4^4,U_1U_3) \in \sigma
\cup \sigma^{-1}$, for every  minimal presentation $\sigma$. Obviously, this
pair does not have full support, and hence $\mathcal B (G^{\bullet})$
has no generic presentation.

Let $G = C_2 \oplus C_2 \oplus C_2 = \{0, e_1, e_2, e_3, e_1+e_2, e_1+e_3, e_2+e_3, e_1+e_2+e_3 \}$. Then $U_1= e_1^2, U_2 = e_2^2, U_3 = (e_1+e_2)^2, V = e_1e_2(e_1+e_2) \in \mathcal A \big( \mathcal B (G^{\bullet}) \big)$ and $(U_1U_2, V^3) \in \ \sim_{\mathcal B (G^{\bullet})}$. Since every  minimal presentation $\sigma$ contains the relation $(U_1U_2, V^3)$ which does not have full support, it follows that $\mathcal B (G^{\bullet})$
has no generic presentation.
\end{proof}

\bigskip
\section{Unions of sets of lengths} \label{6}
\bigskip

\medskip
\begin{definition} \label{6.1}
Let  $S$  be  atomic  and  $k \in \mathbb N$.
\begin{enumerate}
\item If \ $S = S^{\times}$,  we set \ $\mathcal V_k(S) = \{k\}$. If \ $S \ne S^{\times}$,
      let \ $\mathcal V_k (S)$ \ denote the
      set of all \ $m \in \mathbb N$ \ \ for which there exist \ $u_1,
      \ldots, u_k, v_1, \ldots, v_m \in \mathcal A (S)$ with
      $u_1 \cdot \ldots \cdot u_k = v_1 \cdot \ldots \cdot v_m$.

\smallskip
\item We define
      \[
      \rho_k (S) = \sup \mathcal V_k (S) \in \mathbb N \cup \{\infty \}
      \quad \text{and} \quad \lambda_k (S) = \min \mathcal V_k (S) \in
      [1,k] \,.
      \]

\smallskip
\item For $a \in S$, $\rho (a) = \rho \big(\mathsf L (a) \big)$ is called the {\it elasticity} of $a$, and
      \[
      \rho (S) = \sup \{ \rho (L) \mid L \in \mathcal L (S) \} \in \mathbb R_{\ge 1} \cup \{\infty\}
      \]
      is called the {\it elasticity} of $S$.      We say that $S$ has \ {\it finite accepted elasticity } \ if there exists
      some $a \in S$ with $\rho(a)=\rho(S) < \infty$.
\end{enumerate}
\end{definition}

\medskip
Let  $k, l \in \mathbb N$. Then $k \in \mathcal V_k (S)$, $\mathcal
V_k (S) + \mathcal V_l (S) \subset \mathcal V_{k+l} (S)$,
\[
\lambda_{k+l} (S) \le \lambda_k (S) + \lambda_l (S) \le  k + l \le
\rho_k (S) + \rho_l (S) \le \rho_{k+l} (S) \,, \] and
\[
\rho (S) = \sup \Bigl\{ \frac{\rho_{k} (S)}{k} \; \Bigm| \; k \in \N
\Bigr\} = \lim_{k \to \infty}\frac{\rho_k(S)}{k} \quad \text{and}
\quad \frac{1}{\rho (S)} = \inf \Bigl\{ \frac{\lambda_{k} (S)}{k} \;
\Bigm| \; k \in \N \Bigr\}
 = \lim_{k \to \infty} \frac{\lambda_k (S)}{k}
 \,,
\]
(see \cite[Proposition 1.4.2]{Ge-HK06a} and \cite[Section
3]{Ga-Ge09b}). Moreover, if \ $S \ne S^{\times}$, then
\[
      \mathcal V_k (S) \ = \ \bigcup_{k \in L, L \in \mathcal L (S)}
      L
      \]
is the union of all sets of lengths containing $k$. These unions
were introduced by S.T. Chapman and W.W. Smith in \cite{Ch-Sm90a}.
 It was proved only
recently that a $v$-noetherian monoid, which satisfies $\rho_k (S) <
\infty$ for all $k \in \mathbb N$, is locally tame (see
\cite[Corollary 4.3]{Ge-Ha08a}). For Krull monoids with finite class
group, the invariants $\rho_k (S)$ are studied in \cite{Ge-Gr-Yu12}.

The first part of this section  is devoted to the invariant $\rho_k
(S)$ in a more general setting, and after that we study the structure of the unions of sets of lengths for numerical monoids.

\medskip
\begin{proposition} \label{6.2}
Let $S$ be  atomic with $S \ne S^{\times}$.
\begin{enumerate}
\item If $S$ has finite accepted elasticity, then the sets
      \[
      M = \{ k \in \mathbb N \mid \frac{\rho_k (S)}{k} = \rho (S) \} \cup
      \{0\} \quad \text{and} \quad M' = \{ k \in \mathbb N \mid \frac{\lambda_k (S)}{k} = \frac{1}{\rho (S)} \} \cup
      \{0\}
      \]
      are submonoids of
      $(\mathbb N_0, +)$, distinct from $\{0\}$.

\smallskip
\item Let $a \in S$,  $z = u_1 \cdot \ldots \cdot u_l \in
      \mathsf Z (a)$ and $z' = v_1 \cdot \ldots \cdot v_{\rho} \in \mathsf
      Z (a)$ where $l \in \mathbb N$, $\rho = \rho_l (S)$ and $u_1 \ldots,
      u_l, v_1, \ldots, v_{\rho} \in \mathcal A (S_{\red})$. If there is
      no $k \in [1, l-1]$ such that $\rho_k (S) + \rho_{l-k}(S) = \rho_l
      (S)$, then $(z, z') \in \mathcal A (\sim_S)$.
\end{enumerate}
\end{proposition}

\begin{proof}
We may suppose that $S$ is reduced.

\smallskip
1. Suppose that $S$ has finite accepted elasticity. First we
consider the set $M$. By definition, there is an $a \in S$ such that
$\rho (S) = \rho (a)$. If $k = \min \mathsf L (a)$ and $\rho = \max
\mathsf L (a)$, then
\[
\rho (S) = \frac{\rho}{k} \le \frac{\rho_k (S)}{k} \le \rho (S) \,,
\]
and hence $k \in M$. Let $i \in [1,2]$ and $k_i \in M$. Since
$(k_1+k_2)\rho (S) = \rho_{k_1}(S) + \rho_{k_2}(S) \le
\rho_{k_1+k_2}(S)$, it follows that
\[
\rho (S) \ge \frac{\rho_{k_1+k_2}(S)}{k_1+k_2} \ge
\frac{\rho_{k_1}(S) + \rho_{k_2}(S)}{k_1+k_2} = \rho (S) \,.
\]
Thus equality holds, and $k_1+k_2 \in M$. To verify the assertion on
$M'$, we choose an $l \in \mathbb N$ such that $\rho_l (S)/l = \rho
(S)$. Then $\lambda_{\rho_l (S)} (S) \le l$, and since
\[
\frac{1}{\rho (S)} \le \frac{\lambda_{\rho_l (S)}(S)}{\rho_l (S)}
\le \frac{l}{\rho_l (S)} = \frac{1}{\rho (S)} \,,
\]
it follows that $\rho_l (S) \in M'$. Let $i \in [1,2]$ and $k_i \in
M'$. Since $(k_1+k_2)/\rho (S) = \lambda_{k_1}(S) + \lambda_{k_2}
(S) \ge \lambda_{k_1+k_2} (S)$, it follows that
\[
\frac{1}{\rho (S)} \le \frac{\lambda_{k_1+k_2}(S)}{k_1+k_2} \le
\frac{\lambda_{k_1}(S) + \lambda_{k_2}(S)}{k_1+k_2} = \frac{1}{\rho
(S)} \,.
\]
Thus equality holds and $k_1+k_2 \in M'$.

\smallskip
2. Assume to the contrary that $(z, z') \notin \mathcal A (\sim_S)$.
Then there exists $(x, x') \in \ \sim_S$ such that $(x,x') \t (z, z')$
with $1 \ne (x,x') \ne (z, z')$. After renumbering if necessary we may
suppose that $x = u_1 \cdot \ldots \cdot u_k$ and $x' = v_1 \cdot
\ldots \cdot v_{\psi}$ where $k \in [1, l-1]$ and $\psi \in [1,
\rho-1]$. Then $u_{k+1} \cdot \ldots \cdot u_l = v_{\psi+1} \cdot
\ldots \cdot v_{\rho}$ and
\[
\rho_l (S) = \rho = \psi + (\rho-\psi) \le \rho_k (S) +
\rho_{l-k}(S) \le \rho_l (S) \,,
\]
a contradiction.
\end{proof}

\medskip
\begin{corollary} \label{6.3}
Let $S$ be a numerical monoid with $\mathcal A (S) = \{n_1, \ldots,
n_t\}$ where $t \in \mathbb N$ and $1 < n_1 < \nolinebreak \ldots < \nolinebreak n_t$.
\begin{enumerate}
\item Then
\[
\rho (S) = \frac{n_t}{n_1} \quad \text{and} \quad \min \Delta (S) =
\gcd ( n_2-n_1, \ldots, n_t - n_{t-1} ) \,.
\]

\smallskip
\item $\{ k \in \mathbb N \mid \frac{\rho_k (S)}{k} = \rho (S) \} \cup \{0\} = \frac{\lcm(n_1, n_t)}{n_t} \mathbb N_0$.

\smallskip
\item $\{ k \in \mathbb N \mid \frac{\lambda_k (S)}{k} = \frac{1}{\rho (S)} \} \cup \{0\} = \frac{\lcm(n_1, n_t)}{n_1} \mathbb N_0$.
\end{enumerate}
\end{corollary}

\begin{proof}
1. See \cite[Theorem 2.1]{Ch-Ho-Mo06} and \cite[Proposition
2.9]{B-C-K-R06}.

\smallskip
2. Let $a \in \mathbb N$ be a multiple of $\lcm (n_1, n_t)$. We show
that $a/n_t$ is in the set on the left hand side. We have
\[
a = \frac{a}{n_1} \boldsymbol{n_1} = \frac{a}{n_t} \boldsymbol{n_t},
\quad \min \mathsf L (a) \le \frac{a}{n_t}, \quad \max \mathsf L (a)
\ge \frac{a}{n_1}
\]
and
\[
\frac{n_t}{n_1} = \rho (S) \ge \rho (a) = \frac{\max \mathsf L
(a)}{\min \mathsf L (a)} \ge \frac{n_t}{n_1} \,.
\]
This shows that $\min \mathsf L (a) = a/n_t$, $\max \mathsf L (a) =
a/n_1$ and
\[
\frac{n_t}{n_1} = \rho (a) \le \frac{\rho_{\min \mathsf L
(a)}(S)}{\min \mathsf L (a)} \le \rho (S) = \frac{n_t}{n_1} \,.
\]
Thus equality holds and $\min \mathsf L (a) = a/n_t$ has the
required property.

Conversely, let $k \in \mathbb N$ with $\rho_k (S)/k = \rho (S) =
n_t/n_1$. We choose $a \in \mathcal V_k (S)$ with $\max \mathsf L
(a) = \rho_k (S)$. Then $\min \mathsf L (a) \le k$ and
\[
\frac{n_t}{n_1} = \frac{\rho_k (S)}{k} \le \frac{\max \mathsf L
(a)}{\min \mathsf L (a)} = \rho (a) \le \frac{n_t}{n_1}
\]
implies that $\min \mathsf L (a) = k$. Since $a/n_t \le \min \mathsf
L (a)$, $\max \mathsf L (a) \le a/n_1$ and
\[
\frac{n_t}{n_1} = \rho (a) \le \frac{a/n_1}{a/n_t} = \frac{n_t}{n_1}
\,,
\]
it follows that $n_1 \t a$ and $n_t \t a$. Therefore $\lcm (n_1,
n_t) \t a$ and
\[
k = \min \mathsf L (a) = \frac{a}{n_t} \in \frac{\lcm (n_1,
n_t)}{n_t} \mathbb N_0 \,.
\]

\smallskip
3. Let $a \in \mathbb N$ be a multiple of $\lcm (n_1, n_t)$. We show
that $a/n_1$ is in the set on the left hand side which runs along
the lines of 2. Conversely, let $k \in \mathbb N$ with $\lambda_k
(S)/k = 1/\rho (S) = n_1/n_t$. We choose $a \in \mathcal V_k (S)$
with $\min \mathsf L (a) = \lambda_k (S)$. Again arguing as in 2.,
we infer that
\[
k = \max \mathsf L (a) = \frac{a}{n_1} \in \frac{\lcm (n_1,
n_t)}{n_1} \mathbb N_0 \,. \qedhere
\]
\end{proof}

\medskip
\begin{corollary} \label{6.4}
Let  $S$  be a reduced Krull monoid, $F = \mathcal F (P)$ a free monoid such that $S \subset F$ is a saturated and cofinal submonoid,
$G = F/S$ and $G_P = \{ p \, \mathsf q (S) \mid p \in P\} \subset G$ the set of classes containing prime divisors. Suppose that $G_P = -G_P$ and that $\mathsf D (G_P) < \infty$.
\begin{enumerate}
\item We have $\rho (S) = \mathsf D (G_P)/2$ and $2 \mathbb N \subset \{ k \in \mathbb N \mid \frac{\rho_k (S)}{k} = \rho (S) \}$.

\smallskip
\item Let $m \in \N$ be minimal such
      that
      \[
      \rho_{2m+1}(G_P) - m \mathsf D (G_P) = \max \{ \rho_{2k+1}(G_P) - k
      \mathsf D (G_P) \mid k \in \N \} \,.
      \]
      Then $\rho_{2m+1} (S) \le {\mathsf a \big( \mathcal B (G_P) \big) }$.
\end{enumerate}
\end{corollary}

\begin{proof}
1. See \cite[Theorem 3.4.10]{Ge-HK06a}.

\smallskip
2. For every $k \in \mathbb N$, we set (as it is usual) $\rho_k (G_P) = \rho_k \big( \mathcal B (G_P) \big)$,  and by
\cite[Theorem 3.4.10]{Ge-HK06a} we have $\rho_k (S) = \rho_k (G_P)$.
Thus it
suffices to verify that $\rho_{2m+1} (G_P)$ has the asserted upper
bound. Let $U_1, \ldots, U_{2m+1}, V_1, \ldots, V_{\rho} \in
\mathcal A \big( \mathcal B (G_P) \big)$ with $U_1 \cdot \ldots \cdot U_{2m+1} = V_1 \cdot
\ldots \cdot V_{\rho}$ and $\rho = \rho_{2m+1} (G_P)$. We assert
that there is no $k \in [1, 2m]$ such that $\rho_k (G_P) +
\rho_{2m+1-k} (G_P) = \rho$. If this holds, then Proposition
\ref{6.2} implies that $(z= U_1 \cdot \ldots \cdot U_{2m+1}, z' =
V_1 \cdot \ldots \cdot V_{\rho}) \in \mathcal A (\sim_{\mathcal B
(G_P)})$ and hence
\[
\begin{aligned}
\rho_{2m+1} (G_P) & = \rho = \max \{2m+1, \rho\} = \max \{ |z|, |z'| \} \\
& \le \sup \big\{ |x| \mid (x,y) \in \mathcal A (
      \sim_{\mathcal B (G_P)}) \ \text{for some} \ y
      \in \mathsf Z \big( \mathcal B (G_P) \big)  \big\}  = {\mathsf a \big( \mathcal B (G_P) \big) } \,.
\end{aligned}
\]
Assume to the contrary, that there is a $k \in [1, 2m]$ such that
$\rho_k (G_P) + \rho_{2m+1-k} (G_P) = \rho_{2m+1}(G_P)$. Then either $k$
or $2m+1-k$ are odd, say $k = 2s+1$ with $s \in \N_0$. Since, by 1., we have $\rho_{2(m-s)}(G_P) = (m-s)\mathsf D (G_P)$, we infer that
\[
\rho_{2s+1}(G_P) - s \mathsf D (G_P) = \rho_{2m+1} (G_P) -
\rho_{2(m-s)}(G_P) - s \mathsf D (G_P) = \rho_{2m+1} (G_P) - m \mathsf D
(G_P) \,,
\]
a contradiction.
\end{proof}

\medskip
Let all notations be as Corollary \ref{6.4},  and suppose in
addition that $G_P = G$ is finite abelian. In all situations studied
so far, the set
\[
M = \{ k \in \mathbb N \mid \frac{\rho_k (S)}{k} = \rho (S) \} \cup \{0\}
\]
contains an odd element, and hence (by Proposition \ref{6.2} and by Corollary \ref{6.4}.1) $M$ is a numerical monoid. The
standing conjecture is that this holds for all finite abelian groups $G$ (see
\cite{Ge-Gr-Yu12}).

\medskip
Next we deal with the structure of the
unions of sets of lengths. Suppose $S$ is a Krull monoid such that
every class contains a prime divisor. Then it was shown
only recently that, for all $k \in \mathbb N$, the unions $\mathcal
V_k (S)$ are arithmetical progressions with difference $1$ (see
\cite[Theorem 4.1]{Fr-Ge08}, \cite{Ge09a} for a simpler proof, and
also \cite{Ga-Ge09b}). In \cite{Ph13a},  unions of sets of
lengths are studied for non-principal order in number fields, and in  \cite{Ch-Go-Pe01}, for domains of the form $V + XB[X]$, where $V$
is a discrete valuation domain and $B$ the ring of integers in a
finite extension field over the quotient field of $V$. In
\cite{A-C-H-P07a}, S.T. Chapman et al. showed that in numerical
monoids, generated by arithmetical progressions, all unions are
arithmetical progressions. We are going  to generalize
this result.

\medskip
\begin{proposition} \label{6.5}
Let $S$ be a numerical monoid with $\mathcal A (S) = \{n_1, \ldots,
n_t\}$, where $t \in \mathbb N$, $1 < n_1 < \ldots < n_t$,  and $d =
\gcd ( n_2-n_1, \ldots, n_t - n_{t-1} )$. Suppose that the
Diophantine equations
\[
(n_2-n_1) x_2 + \ldots + (n_t-n_1)x_t = dn_1 \quad \text{and} \quad
(n_t-n_1)y_1 + \ldots + (n_t- n_{t-1})y_{t-1} = d n_t
\]
have solutions in the non-negative integers. Then there exists an
element $a^* \in S$ such that $\rho (a^*) = \rho (S)$ and $\mathsf L
(a^*) $ is an arithmetical progression with difference $d$.
\end{proposition}

\begin{proof}
We proceed in several steps.

{\bf 1.} Let $a \in \mathbb N$ be a multiple of $n_1$ and of $n_t$.
Then
\[
z = \frac{a}{n_1} {\boldsymbol {n_1}} \quad \text{and} \quad  z' =
\frac{a}{n_t} {\boldsymbol {n_t}}
\]
are factorizations of $a$. Obviously, we have $\min \mathsf L (a) =
a/n_t$, $\max \mathsf L (a) = a/n_1$ and hence $\rho (a) = n_1^{-1}
n_t$. By Corollary \ref{6.3}.1 it follows that $\rho (a) = \rho
(S)$.

{\bf 2.} Since $S$ is finitely generated,  Proposition \ref{5.2}.2 and Equation \ref{basic} imply that  $S$ is
locally tame with finite  set of distances $\Delta (S)$, and $\Delta (S) \ne \emptyset$ because $\rho (S) > 1$.
Thus \cite[Theorem 4.3.6.1]{Ge-HK06a} implies that there is an
$\bar{a} \in S$ with the following property: for every $b \in S$ we
have
\[
\mathsf L (\bar{a}b) = y + (L' \cup L^* \cup L'') \subset y + d \Z
\]
where $y \in \mathbb Z$, $L^*$ is an arithmetical progression with
difference $d$, $\min L^* = 0$, $L' \subset \big[- \mathsf t \big(S, \mathsf
Z (\bar{a}) \big), -1 \big]$ and $L'' \subset \max L^* + \big[1, \mathsf t \big(S,
\mathsf Z (\bar{a}) \big) \big]$.

{\bf 3.} Let $(\alpha_2, \ldots, \alpha_t) \in \mathbb N_0^{t-1}$
and $(\beta_1, \ldots, \beta_{t-1}) \in \mathbb N_0^{t-1}$ be
solutions of the given Diophantine equations, and set
\[
\alpha_1 = \alpha_2 + \ldots + \alpha_t \quad \text{and} \quad
\beta_t = - (\beta_1+ \ldots + \beta_{t-1}) \,.
\]
Now let $a^* \in \mathbb N$ be a multiple of  $\lcm (\bar{a}, n_1,
n_t)$ such that
\[
\frac{a^*}{n_1}  \ge \gamma (d + \alpha_1) \quad \text{and} \quad
\frac{a^*}{n_t}  \ge \gamma |d + \beta_t| \quad \text{where} \quad
\gamma = \Big\lceil \frac{ \mathsf t \big(S, \mathsf Z (\bar{a}) \big)}{d}
\Big\rceil \,.
\]
We assert that $a^*$ has the required properties. By {\bf 1.}, it
follows that
\[
\min \mathsf L (a^*) = \frac{a^*}{n_t}, \ \max \mathsf L (a^* ) =
\frac{a^*}{n_1} \quad \text{and} \quad \rho (a^* ) = \rho (S) \,.
\]
We set $a^* = \bar{a} b$ with $b \in S$, and write $\mathsf L (a^*)$
in the form $\mathsf L (a^*) = y + (L' \cup L^* \cup L'') \subset y
+ d \mathbb Z$ with all properties as in {\bf 2.} (note that such a
representation need not be unique).

Let $\nu \in [0, \gamma]$. Then
\[
x_{\nu} = \Big( \frac{a^*}{n_1} - \nu(d+\alpha_1)\Big){\boldsymbol
{n_1}} + \nu \alpha_2 {\boldsymbol {n_2}} + \ldots + \nu \alpha_t
{\boldsymbol {n_t}}
\]
is a factorization of $a^*$ of length
\[
|x_{\nu}| = \frac{a^*}{n_1}  - \nu d - \nu (\alpha_1 - \alpha_2 -
\ldots - \alpha_t) = \max \mathsf L(a^*) - \nu d \in \mathsf L (a^*)
\,.
\]
Similarly,
\[
y_{\nu} = \Big( \frac{a^*}{n_t} + \nu(d + \beta_t) \Big){\boldsymbol
{n_t}} + \nu \beta_1 {\boldsymbol {n_1}} + \ldots + \nu  \beta_{t-1}
{\boldsymbol {n_{t-1}}}
\]
is a factorization of $a^*$ of length
\[
|y_{\nu}| = \frac{a^*}{n_t} + \nu d + \nu (\beta_1+ \ldots +
\beta_{t-1}+\beta_t) = \min \mathsf L (a^*) + \nu d \in \mathsf L
(a^*) \,.
\]
This reveals that $\mathsf L (a^*)$ starts and ends with
arithmetical progressions having difference $d$ and $(\gamma + 1)$
elements. Thus it follows that $L'$ and $L''$ are  (possibly empty)
arithmetical progressions with difference $d$, and thus $\mathsf L
(a^*)$ is an arithmetical progression with difference $d$.
\end{proof}

\medskip
\begin{theorem} \label{6.6}
Let $S$ be a numerical monoid with $\mathcal A (S) = \{n_1, \ldots,
n_t\}$ where $t \in \mathbb N$, $1 < n_1 < \ldots < n_t$,  and $d =
\gcd ( n_2-n_1, \ldots, n_t - n_{t-1} )$. Suppose that the
Diophantine equations
\[
(n_2-n_1) x_2 + \ldots + (n_t-n_1)x_t = dn_1 \quad \text{and} \quad
(n_t-n_1)y_1 + \ldots + (n_t- n_{t-1})y_{t-1} = d n_t
\]
have solutions in the non-negative integers. Then there exists a $k^*
\in \mathbb N$ such that $\mathcal V_k (S)$ is an arithmetical
progression with difference $d$ for all $k \ge k^*$, and
\[
      \lim_{k \to \infty} \frac{|\mathcal V_k (S)|}{k} = \frac{1}{d}
      \Bigl( \frac{n_t}{n_1} - \frac{n_1}{n_t} \Bigr) \,.
      \]
\end{theorem}

\begin{proof}
By Proposition \ref{6.5}, all assumptions in \cite[Theorem
3.1]{Fr-Ge08} are satisfied, and hence this result implies the
assertion.
\end{proof}

\medskip
\begin{remarks} \label{6.7}~

\smallskip
1. If $\mathcal A (S)$ is an arithmetical progression, then all sets
$\mathcal V_k (S)$ are arithmetical progressions (see \cite[Theorem
2.7]{A-C-H-P07a}). However, in general, we have $k^*
> 2$. Indeed,  $S = \langle 4, 5, 13, 14 \rangle$ satisfies the assumptions of Theorem \ref{6.6}, but since $\mathcal
V_2 (S) = \{2, 6, 7\}$ is not an arithmetical progression, it follows that $k^* > 2$.

\smallskip
2. Unions of sets of lengths in finitely generated monoids  are
almost arithmetical progressions (see \cite[Theorems 3.5 and
4.2]{Ga-Ge09b}). But even in a numerical monoid, there may exist infinitely many $k \in \mathbb N$, for which these unions are not  arithmetical
progressions, as the following example shows.

Let $S = \langle 4,10, 21 \rangle$ and $k \in \mathbb N$. Then $d =
\gcd (6, 11) = 1$. We assert that  $\mathcal V_k (S)$ is not an
arithmetical progression with difference $1$. We set $S_k = \{ a \in
S \mid k \in \mathsf L (a) \}$ and observe that
\[
S_k = \{ a\mathbf{4}+b\mathbf{10}+c\mathbf{21} \mid a, b, \ c \in
\mathbb N_0 \ \text{with} \ a+b+c=k \}, \ \min S_k = 4k \quad
\text{and} \quad \max S_k = 21k \,.
\]
In particular, we see that $S_k = \{ 4k, \ldots, 21k-28, 21k-22,
21k-17, 21k-11, 21k \}$, where the elements are written down in
increasing order. The element $21k$ has a unique factorization of
maximal length, namely
\[\left\{
\begin{array}{ll}
21t\mathbf{4} & \hbox{if } k=4t,\\
21t\mathbf{4}+\mathbf{21} & \hbox{if } k=4t+1,\\
(21t+8)\mathbf{4}+\mathbf{10} & \hbox{if } k=4t+2,\\
(21t+8)\mathbf{4}+\mathbf{10}+\mathbf{21} & \hbox{if } k=4t+3.
\end{array}
\right.
\]
Setting $l = \max \mathsf L (21k)$ we assert that there is no $s \in
S_k$ with $l-1 \in \mathsf L (s)$. If this holds, then $\mathcal V_k (S)$ is not an
arithmetical progression with difference $1$. To verify our assertion we distinguish four cases.
\begin{itemize}
\item If $k=4t$, then $l=21t$. An element with a factorization of length $21t-1$ is greater than or equal to $(21t-1)4=(21t)4-4>21k-11$, and thus it does not belong to $S_k$.

\smallskip
\item If $k=4t+1$, then $l=21t+1$. Elements having a factorization of length $21t$ are $(21t)4$, $(21t-1)4+10=(21t)4+6$, $(21t-2)4+2\cdot 10=21t+12$, $(21t-1)4+21=(21t)4+17$, $(21t-2)4+10+21=(21t)4+23$,\ldots. In this setting the four largest elements of $S_k$ are $(21t)4+21$, $(21t)4+10$, $(21t)4+4$ and $(21t)4-1$. Hence also in this case, there is no element in $S_k$ having a factorization of length $l-1$.

\smallskip
\item If $k=4t+2$, then $l=21t+9$. The set of elements having a factorization of length $21t+8$ is $\{(21t+8)4=(21t)4+32, (21t+7)4+10=(21t)4+39,\ldots,(21t+8)21\}$, and the two largest elements of $S_k$ are $(21t)4+42$ and $(21t)4+31$. Again we see that no element in $S_k$ can have a factorization of length $l-1$.

\smallskip
\item If $k=4t+3$, then $l=21t+10$. Arguing as above one easily checks that no elements in $S_k$ have factorizations of length $l-1$.
\end{itemize}
\end{remarks}

\medskip
In view of Theorem \ref{6.6} and the Remarks \ref{6.7} we end this paper with the formulation of the following problem.

\medskip
\noindent
{\bf Open Problem.} Characterize the numerical monoids $S$ for which there exists a $k^* \in \mathbb N$ such that the unions of sets of lengths $\mathcal V_k (S)$
are arithmetical progressions for all $k \ge k^*$.

\bigskip

% \bibliography{ger,hk,fact,zerosum,ideal}
% \bibliographystyle{amsplain}

\providecommand{\bysame}{\leavevmode\hbox to3em{\hrulefill}\thinspace}
\providecommand{\MR}{\relax\ifhmode\unskip\space\fi MR }
\providecommand{\MRhref}[2]{%
  \href{http://www.ams.org/mathscinet-getitem?mr=#1}{#2}
}
\providecommand{\href}[2]{#2}

\end{document}